%% file: main.tex
\documentclass[DIV=13,abstract=true,paper=a4,fontsize=11pt,parskip=half]{scrartcl}

\RequirePackage{booktabs}
\RequirePackage{dsfont}
\RequirePackage{microtype}
\usepackage{lmodern}
\usepackage{inputenc}
\usepackage[T1]{fontenc}
\usepackage[USenglish]{babel}
\usepackage{latexsym}
\usepackage{amsmath,amsfonts,amssymb,amsthm,amsopn}
\usepackage[centercolon]{mathtools}
\usepackage{fixmath}
\usepackage{braket}
\usepackage{graphicx}
\usepackage{enumerate}
\usepackage[svgnames]{xcolor}
\usepackage[pdftex,bookmarks=true,pdfpagelabels=true,plainpages=false,
hypertexnames=false,bookmarksnumbered=true,bookmarksopen=true,
pdfstartview=Fit,colorlinks,
linkcolor=DarkBlue,citecolor=DarkGreen,urlcolor=DarkRed]{hyperref}
\usepackage[algo2e,ruled]{algorithm2e}
\SetKw{KwTerminate}{terminate}

\SetCommentSty{mycommfont}
\usepackage[nameinlink,capitalise]{cleveref}
\crefformat{equation}{(#2#1#3)}
\crefname{assumption}{Assumption}{Assumptions}
\Crefname{assumption}{Assumption}{Assumptions}
\crefname{figure}{Figure}{Figures}
\theoremstyle{plain}
\newtheorem{theorem}{Theorem}[section]
\newtheorem{lemma}[theorem]{Lemma}
\newtheorem{corollary}[theorem]{Corollary}
\theoremstyle{definition}

\newtheorem{assumption}[theorem]{Assumption}

\theoremstyle{remark}
\newtheorem{remark}[theorem]{Remark}

\providecommand{\AMS}[1]{\textbf{AMS subject classifications.} #1}

\DeclareMathOperator{\bild}{ran}
\newcommand{\calA}{\mathcal{A}} 
\newcommand{\calW}{\mathcal{W}} 
\newcommand{\spann}[1]{\ensuremath{\langle #1\rangle}}
\newcommand{\tol}{\mathfrak{e}} 
\newcommand{\Ballop}{\mathbb{B}} 
\newcommand{\R}{\mathbb{R}} 
\newcommand{\N}{\mathbb{N}} 
\newcommand{\uopt}{\ensuremath{\bar u}}
\newcommand{\norm}[2]{\ensuremath{\lVert #1\rVert}}

\title{On the convergence of Broyden's method and some accelerated schemes for singular problems}
\author{Florian Mannel\thanks{University of Lübeck, Maria-Göppert-Straße 3, 23562 Lübeck, Germany (\href{mailto:mannel@mic.uni-luebeck.de}{mannel@mic.uni-luebeck.de}).}\, \href{https://orcid.org/0000-0001-9042-0428}{\includegraphics[height=.35cm]{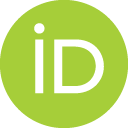}}}

\date{Preprint, \today}

\ifpdf
\hypersetup{
	pdftitle={On the convergence of Broyden's method and some accelerated schemes for singular problems},
	pdfauthor={F. Mannel}
}
\fi

\begin{document}

\maketitle

\begin{abstract}
We consider Broyden's method and some accelerated schemes for nonlinear equations having a strongly regular singularity of first order with a one-dimensional nullspace. Our two main results are as follows. First, we show that the use of a preceding Newton--like step ensures convergence for starting points in a starlike domain with density~1. This extends the domain of convergence of these methods significantly.
Second, we establish that the matrix updates of Broyden's method converge q-linearly with the same asymptotic factor as the iterates. This contributes to the long--standing question whether the Broyden matrices converge by showing that this is indeed the case for the setting at hand. 
Furthermore, we prove that the Broyden directions violate uniform linear independence, which implies that existing results for convergence of the Broyden matrices cannot be applied.
Numerical experiments of high precision confirm the enlarged domain of convergence, the q-linear convergence of the matrix updates, and the lack of uniform linear independence. In addition, they suggest that these results can be extended to singularities of higher order and that Broyden's method can converge r-linearly without converging q-linearly. The underlying code is freely available.
\end{abstract}


\AMS{49M15, 65H10, 65K05,90C30, 90C53}

%%%%%%%%%%%%%%%%%%%%%%%%%%%%%%%%%%%%%%%%%%%%%%%%%%%%%%%%%%%%%%%%%%%%%%%%%%%%%%%%%
%%%%%%%%%%%%%%%%%%%%%%%%%%%%%%%%%%%%%%%%%%%%%%%%%%%%%%%%%%%%%%%%%%%%%%%%%%%%%%%%%

\section{Introduction}

\subsection{Broyden's method for singular equations}

Broyden's method \cite{Broyden} is one of the most widely used quasi-Newton methods for solving systems of nonlinear equations $F(u)=0$ with $F:\R^n\rightarrow\R^n$, see \cite{Book_DennisSchnabel,Kelley,Martinez_00,NocWri_06,GriewankHistoricalRemarksOnQuasiNewtonMethods}. It reads as follows.

\pagebreak

\begin{algorithm2e}[h!]
	\SetAlgoRefName{BM}
	\DontPrintSemicolon
	\caption{Broyden's method}
	\label{alg_broy}
	\KwIn{ $u^0\in\R^n$, $B_0\in\R^{n\times n}$ \text{invertible}
	}
	\For{ $k=0,1,2,\ldots$ }
	{   
		\lIf{$F(u^k)=0$ }{let $u^\ast:=u^k$; STOP}
		Solve $B_k s^k = -F(u^k)$ for $s^k$
		\tcp*{Broyden step}
		Let $u^{k+1}:=u^k+s^k$\\
		Let $y^k:=F(u^{k+1})-F(u^k)$\\
		Let $B_{k+1}:=B_k+(y^k-B_k s^k)\frac{(s^k)^T}{\norm{s^k}{}^2}$		
		\tcp*{Broyden update}
	}
	\KwOut{ $u^\ast$ }
\end{algorithm2e}

This work concerns Broyden's method for \emph{singular equations}. 
Specifically, we consider smooth mappings $F$ that admit a root $\uopt$ with the following properties:
\begin{itemize}
	\item There is $\phi\in\R^n$ with $\norm{\phi}{2}=1$ such that 
	$N:=\ker(F'(\uopt))=\spann{\phi}$ and 
	such that $X:=\bild(F'(\uopt))$ satisfies 
	$X\oplus N=\R^n$.
	\item There holds $P_N (F''(\uopt)(\phi,\phi))\neq 0$, where $P_N\in\R^{n\times n}$ denotes the projection onto $N$ parallel to $X$. 
\end{itemize}
Using these assumptions Decker and Kelley established in \cite{DeckerKelley_BroydenSingularJacobian} that the iterates 
$(u^k)$ of Broyden's method converge to $\uopt$ if $\norm{u^0-\uopt}{}$, $\norm{B_0-F'(u^0)}{}$, and 
the angle between $u^0-\uopt$ and $N$ are sufficiently small. 
The main contributions of this article complement their result as follows: 
\begin{itemize}
	\item We show that the restriction on the angle can be removed almost completely if a Newton--like step is taken before starting Broyden's method. That is, the domain of convergence of Broyden's method can be extended significantly by a preceding Newton--like step; cf. Algorithm~\ref{alg_broypNl}. (In fact, we demonstrate this also for some \emph{accelerated Newton--type methods}.)
	
	\item We prove that the assumptions used by Decker and Kelley to show convergence of $(u^k)$ also imply convergence of the Broyden matrices $(B_k)$. (In fact, we establish this under weaker assumptions that do not require singularity of $F'(\uopt)$.)
\end{itemize}
Before we discuss these contributions further, let us stress that it is a long--standing open question whether the
Broyden matrices converge, cf. the survey articles  
\cite[Example~5.3]{QuasiNewtonMethodsReview_DennisMore},
\cite[p.~117]{Martinez_00}, \cite[p.~306]{GriewankHistoricalRemarksOnQuasiNewtonMethods} and \cite[p.~940]{AlBaaliSpedicatoMaggioni_BroyReview}. 
For the class of singular problems under consideration, this work provides an affirmative answer.

\begin{algorithm2e}
	\SetAlgoRefName{BMP}
	\DontPrintSemicolon
	\caption{Broyden's method with a preceding Newton--like step}
	\label{alg_broypNl}
	\KwIn{ $\hat u\in\R^n$, $\hat B\in\R^{n\times n}$ \text{invertible}
	}
	Let $u^0:=\hat u - \hat B^{-1} F(\hat u)$
	\tcp*{preceding Newton-like step}
	Select invertible $B_0\in\R^{n\times n}$
	\tcp*{initial guess for Broyden matrix}
	Call Algorithm~\ref{alg_broy} with $(u^0,B_0)$\\ 	
	\KwOut{ $u^\ast$ (generated by Algorithm~\ref{alg_broy})}
\end{algorithm2e}

\subsection{Enlarging the domain of convergence through a preceding Newton--like step}

\subsubsection{Broyden's method}\label{sec_introenlarging}

Decker and Kelley show in \cite{DeckerKelley_BroydenSingularJacobian} that for singularities satisfying the above two assumptions, the iterates of Broyden's method converge q-linearly to $\uopt$ for starting points $u^0$ in the intersection of a ball and a double cone with axis $\uopt+N$, i.e., for
\begin{equation}\label{def_theta}
	u^0 \in W_{\uopt}(\rho,\theta) := \Bigl\{
	u\in\R^n: \enspace \norm{P_X(u-\uopt)}{2}\leq \theta\norm{P_N(u-\uopt)}{2}\Bigr\}\;\cap\;\bar{\mathbb{B}}_{\rho}^\ast(\uopt)
\end{equation}
with sufficiently small $\rho,\theta>0$, provided $B_0$ is close to $F'(u^0)$. Here, 
$P_X\in\R^{n\times n}$ denotes the projection $P_X:=I-P_N$ onto $X$, and $\bar{\mathbb{B}}_{\rho}^\ast(\uopt)$ is the deleted closed ball of radius $\rho$ centered at $\uopt$. 
However, the set $W_{\uopt}(\rho,\theta)$
may cover only a small fraction of $\bar{\mathbb{B}}_{\rho}^\ast(\uopt)$ since, to ensure invertibility of $B_0\approx F'(u^0)$, it must be disjoint from the \emph{singular set} $\{u\in\R^n: \det(F'(u))=0\}$ that may form a small angle with $N$ at $\uopt$.
(The singular set is a manifold of dimension $n-1$ transversal to $N$, cf. e.g. the illustration in \cite{DeckerKelley_BroydenSingularJacobian}.) 
For Newton's method, in contrast, Griewank has shown  \cite{Griewank_Diss,Griewank_StarlikedomainsofconvNewtSing,Griewank_StarlikedomainsofconvNewtSingSiamRev} under weaker assumptions on the singularity that the set of starting points that induce q-linear convergence to $\uopt$ is considerably larger in that it includes a \emph{starlike domain with density~1 at $\uopt$}, say $\calW_{\uopt}$. 
We recall from \cite[(2)]{Griewank_StarlikedomainsofconvNewtSing} that $\calW_{\uopt}\subset\R^n$ is \emph{starlike with respect to $\uopt$} iff
\begin{equation*}
	u\in\calW_{\uopt} \qquad\Longrightarrow\qquad \lambda u + (1-\lambda)\uopt\in\calW_{\uopt} \quad\forall\lambda\in(0,1),
\end{equation*}
where $\uopt$ need not belong to $\calW_{\uopt}$. 
Moreover, $\calW_{\uopt}$ has \emph{density $1$ at $\uopt$} iff
\begin{equation*}
	\lim_{r\to 0^+} \frac{L(\calW_{\uopt}\cap\Ballop_r(\uopt))}{L(\Ballop_r(\uopt))} = 1,
\end{equation*}
where $L$ denotes the $n$-dimensional Lebesgue measure and $\Ballop_r(\uopt)$ is the open ball of radius $r$ centered at $\uopt$. 
In the remainder of this work it is understood that the point with respect to which $\calW_{\uopt}$ is starlike or has density~1 is $\uopt$. 
Griewank's result says that the starlike domain $\calW_{\uopt}$ occupies an arbitrarily large portion of $\Ballop_r(\uopt)$ as $r$ tends to zero. A graphical illustration of such a set $\calW_{\uopt}$ is provided in \cite[Section~2]{Griewank_StarlikedomainsofconvNewtSingSiamRev} and in the numerical section of this work.
To establish that result, a crucial part is to construct $\calW_{\uopt}$ with the property that the first Newton step leads into a domain that is shaped similarly to $W_{\uopt}(\rho,\theta)$. It is then shown that all subsequent iterates $u^k$, $k\geq 1$, stay inside this domain and converge q-linearly to $\uopt$. 
Similarly, we demonstrate that the set of points from which a Newton--like step leaps into $W_{\uopt}(\rho,\theta)$ contains a starlike domain with density~1, say again $\calW_{\uopt}$. It follows that for $\hat u\in \calW_{\uopt}$, Algorithm~\ref{alg_broypNl} generates a sequence $(u^k)$ that converges q-linearly to $\uopt$, provided $\hat B$ and $B_0$ are sufficiently close to $F'(\hat u)$, respectively, to $F'(u^0)$. 
This main result, which shows that the domain of convergence of Algorithm~\ref{alg_broypNl} is significantly larger than $W_{\uopt}(\rho,\theta)$,  
is provided in \Cref{thm_starlikedomainofconvforBroyden}. 

\subsubsection{Accelerated methods based on Newton--like steps}

It is possible to accelerate Newton--type methods for singular equations, for instance by use of extrapolation, cf. e.g.  \cite{Griewank_StarlikedomainsofconvNewtSingSiamRev,FischerIzmailovSolodov_NewtonSingularJacobianLineSearch2}. 
Whereas Newton--type methods and in particular Broyden's method are only linearly convergent, acceleration produces superlinearly convergent algorithms. Since some of these \emph{accelerated methods} employ $W_{\uopt}(\rho,\theta)$ as set of starting points, the use of a preceding Newton--like step yields, just as for Broyden's method, convergence for all starting points in a starlike domain with density~1. \Cref{thm_starlikedomainofconvforShamanskii} exemplifies this for an acceleration scheme that combines Newton steps and simplified Newton steps achieving a q-order of convergence of almost $\frac{1+\sqrt5}{2}$.

\subsection{Convergence of the Broyden matrices}\label{sec_introconvBrmatrices}

There are only two results available on the convergence of the Broyden matrices. The first one is a \emph{sufficient condition}: 
In \cite[Theorem~5.7]{MoreTrangenstein} and \cite{convBroydenmatrix} it is shown that if the sequence of normalized Broyden steps $\Bigl(\frac{s^k}{\norm{s^k}{2}}\Bigr)$ is \emph{uniformly linearly independent}, then $(B_k)$ converges and $\lim_{k\to\infty} B_k=F'(\uopt)$, where $\uopt=\lim_{k\to\infty} u^k$. This result holds for singular and for regular $F'(\uopt)$.
We recall from \cite[(AS.4)]{SROneConv_CGT} that a sequence of unit vectors $(v^k)\subset\R^n$ is called \emph{uniformly linearly independent} iff
there exist constants $m\in\N$ and $\rho>0$ such that for every sufficiently large $k\in\N$ the set
\begin{equation*}
	\bigl\{
	v^k, v^{k+1}, \ldots, v^{k+m}
	\bigr\}
\end{equation*}
contains $n$ vectors $v^{k_1}, \ldots, v^{k_n}$ such that all singular values of the $n\times n$ matrix 
\begin{equation*}
	\begin{pmatrix}
		v^{k_1} \, & \, v^{k_2} \, & \, \cdots \, & \, v^{k_n}
	\end{pmatrix}
\end{equation*}
are larger than $\rho$. However, conditions for the uniform linear independence (ULI) of the Broyden steps are unknown for $n\geq 2$ and we are not aware of finite-dimensional examples in which ULI is satisfied.
To the contrary, we show that in the setting under consideration the normalized Broyden steps violate ULI because they align increasingly better with the nullspace $N$. 
Since we also demonstrate that $(B_k)$ converges, this implies that ULI is not a necessary condition for convergence of $(B_k)$. 
In addition, we observe that $\lim_{k\to\infty} B_k=F'(\uopt)$ \emph{never} holds in the numerical experiments, which include an example with regular $F'(\uopt)$.
Moreover, the author of this paper has proven in \cite{Man_convBrmatressoneD} that 
ULI is necessarily violated if the system contains linear equations and the initial matrix $B_0$ is chosen such that its corresponding rows match the coefficients of the linear equations. 
For the special case of a system that contains at least $n-1$ linear equations and $B_0$ is chosen in the way just described, 
convergence of $(B_k)$ is shown in \cite{Man_convBrmatressoneD}, which constitutes the second available result on the convergence of the Broyden matrices. 
In summary, little is known about the convergence of the Broyden matrices and the uniform linear independence of the Broyden steps, and it is one of the main goals of this paper to contribute to this topic. In this respect, we establish the following results:
\begin{itemize}
	\item In \Cref{cor_rlinearrateforupdates} we provide a condition under which r-linear convergence of $(u^k)$ implies r-linear convergence of the Broyden updates $B_{k+1}-B_k$ to zero, and in 
	\Cref{thm_convsingularcase} we show that this condition is satisfied 
	for Algorithm~\ref{alg_broy} if $u^0\in W_{\uopt}(\rho,\theta)$, respectively, for Algorithm~\ref{alg_broypNl} if $\hat u\in\calW_{\uopt}$. 
	It follows that for these starting points the Broyden matrices converge, cf. \Cref{thm_rlinconvofBrupdates}. 
	
	\item In \Cref{lem_qlinearrateofupdates} we provide conditions under which the Broyden updates converge q-linearly to zero, and in \Cref{thm_qlinearrateforupdates} we show that these conditions are satisfied if $u^0\in W_{\uopt}(\rho,\theta)$, respectively, if $\hat u\in\calW_{\uopt}$. 
	
	\item In \Cref{lem_ULIviolated} we establish that if $u^0\in W_{\uopt}(\rho,\theta)$, respectively, if $\hat u\in\calW_{\uopt}$, 
	then ULI is violated.
	
	\item In \Cref{lem_singularityoflimitofBroydenmatrices} we prove that $B:=\lim_{k\to\infty} B_k$ satisfies $N\subset\ker(B)$. In the numerical experiments we will consistently observe $\ker(B-F'(\uopt))=N$, hence in particular $B\neq F'(\uopt)$.
		
\end{itemize} 

\subsection{Numerical experiments}

We conduct numerical experiments with high precision that confirm the theoretical results and suggest two generalizations of the theory. The \href{https://julialang.org/}{Julia} (\cite{JuliaLanguage}) code that underlies the experiments is available at  \href{https://arxiv.org/abs/2111.12393}{https://arxiv.org/abs/2111.12393}. 

\subsection{Related work}

For regular equations, there is an enormous amount of works on Broyden's method, so we can only provide a selection. Broyden, Dennis and Moré have shown the local q-superlinear convergence of several quasi-Newton methods including Broyden's method in \cite{BroydenDennisMore}. Subsequently, Broyden's method has been found to terminate finitely on linear systems \cite{ConvOfBroydensMethodOnLinearSystems_Gay,ConvOfBroydensMethodOnLinearSystemsGaysResultMadeMoreTransparent_OLeary}, 
it has been extended to infinite-dimensional settings \cite{Sachs_FirstSuperlinConvResultInInfiniteDimSpaces,Griewank_87}, and its global convergence based on line searches has been demonstrated \cite{Griewank_globalconvBroyden,LiFukushima2000}.
Recent studies include its extension to constrained nonlinear systems of 
equations \cite{MariniMoriniPorcelli}, to set-valued mappings \cite{BroydensMethodInHilbertSpacesSetValued_ArtachoBelyakovDontchevLopez}, and to single- and set-valued nonsmooth problems in infinite-dimensional spaces \cite{HybridapproachsimilartooursbutInASpecialCaseAndOnlyLinearConvergence_MuoiHaoMaassPidcock,Adly_Ngai_2017,ManRunTwo}.

For singular equations, there are few studies available concerning Broyden's method. 
Besides the aforementioned \cite{DeckerKelley_BroydenSingularJacobian} we are aware of \cite{BuhmilerEtAl}, 
where several quasi-Newton methods and accelerated schemes for regular first order singularities are discussed and studied numerically, 
the recent work \cite{Jarlebring}, where Broyden's method is applied to solve nonlinear eigenvalue problems,
and the brief description of Broyden's method in the context of \emph{bordering singularities} in \cite[end of section~5]{Griewank_StarlikedomainsofconvNewtSingSiamRev}.
In contrast, there are many contributions available concerning singular problems for other Newton--type methods and related accelerated schemes, e.g.
\cite{Gay_NewtonSingular,Gay_NewtonSingularCorrection,DeckerKelley_NewtonSingularJacobian,Griewank_StarlikedomainsofconvNewtSing} for Newton's method, \cite{DeckerKelley_ChordSingularJacobian} for the chord method, \cite{KelleyXue_InexactNewtonSingularJacobian} for inexact Newton methods, \cite{DeckerKelley_AccelerationNewtonSingularJacobian,KelleySuresh_AccelerationNewtonSingularJacobian,Kelley_ShamanskiiLikeAcceleration,ShenYpma_NewtonSingularJacobianAcceleration} for accelerated schemes, \cite{Rall_NewtonMultipleSolutions,Reddien_NewtonSingular,DeckerKellerKelley_NewtonsMethodSingularJacobianIncludingAccelerations} for Newton's method and accelerated schemes, and the survey article \cite{Griewank_StarlikedomainsofconvNewtSingSiamRev}.
More recent works are \cite{OberlinWright_NewtonSemismoothSingularJacobian}, where the smoothness assumptions on the Jacobian are weakened to strong semismoothness,
\cite{IzmailovKurennoySolodov_NewtonSingularJacobianStability,IzmailovKurennoySolodov_NewtonSingularJacobianLocalAttraction,FischerIzmailovSolodov_NewtonSingularJacobianConstrainedEquations}, where various Newton--type methods are considered, \cite{FischerIzmailovSolodov_NewtonSingularJacobianLineSearch1,FischerIzmailovSolodov_NewtonSingularJacobianLineSearch2}, where line-search strategies and acceleration are investigated, and the numerical study \cite{PollockSchwartz_BenchmarkingNewtonAnderson} of the Newton-Anderson method that includes results for singular problems. In \cite{GriewankOsborne_NewtonIrregularSingularities,IzmailovUskov_CriticalLagrangeMultipliers} Newton's method is analyzed for singular optimization problems.

We point out that many of the recent contributions to the theory focus on weakening the assumptions imposed on the singularities, often by using \emph{2-regularity}. Notably, 2-regularity does not require the singularities to be isolated, but still allows to prove that the domain of convergence of various Newton--type methods includes a starlike set that is 
``not too small'', see, e.g.,  \cite{OberlinWright_NewtonSemismoothSingularJacobian,IzmailovKurennoySolodov_NewtonSingularJacobianLocalAttraction,FischerIzmailovSolodov_NewtonSingularJacobianConstrainedEquations}. 
However, under 2-regularity the domain of convergence is generally not of density 1; cf., for instance, the discussion in \cite[p.~368]{IzmailovKurennoySolodov_NewtonSingularJacobianLocalAttraction}.
By following Decker and Kelley we place comparably strong assumptions on the singularity that imply, in particular, that $\bar u$ is isolated. On the plus side, we obtain a domain of convergence with density 1. The convergence of Broyden’s method under weaker assumptions like 2-regularity remains a topic for future research.

Let us also remark that \emph{local error bound conditions}, which are sometimes used to study Newton--type methods, 
cf. e.g. \cite{Griewank_87,BroydensMethodInHilbertSpacesSetValued_ArtachoBelyakovDontchevLopez,DontRock_InexactNewton}, 
do not hold in the setting of this work; see \Cref{rem_errorboundcond} for a rigorous argument. In fact, this is already true under 2-regularity, see the discussion of \emph{criticality} in \cite[Introduction]{IzmailovKurennoySolodov_NewtonSingularJacobianLocalAttraction}.

To the best of my knowledge there are no further results available on the convergence of the Broyden matrices except the two mentioned above. 
For other quasi-Newton methods the following results exist. 
Convergence of the symmetric rank-one (SR$1$) matrices to the true Hessian is shown in \cite{SROneConv_CGT} with ULI as the main assumption and in \cite{SROneConv_KBS} with positive definiteness as the main assumption.
For the self-scaling SR$1$ method convergence of the generated matrices is established in \cite{SROneSelfScalingConv_Sun} with ULI as the main assumption.
Powell shows in \cite{ConvOfPSBMatrices_Powell} that if ULI is satisfied, then the Powell--symmetric--Broyden (PSB) matrices converge to the true Hessian and he proposes algorithmic modifications to ensure that ULI holds.
In \cite{ConvBFGSMatrices_GePowell} it is shown that the Davidon--Fletcher--Powell (DFP) and Broyden--Fletcher--Goldfarb--Shanno (BFGS) matrices converge under very general assumptions and in \cite{ConvOfConvBrClMatrices_Stoer} these results are extended to the convex Broyden class excluding DFP. Notably, ULI is not used in \cite{ConvBFGSMatrices_GePowell,ConvOfConvBrClMatrices_Stoer} and 
convergence to limits other than the true Hessian is possible and actually observed. 

\subsection{Organization and notation}

This paper is organized as follows. In \Cref{sec_conviterates} we extend the domain of convergence of Broyden's method and accelerated schemes through a preceding Newton--like step.  \Cref{sec_main} is devoted to the convergence of the Broyden matrices, while \Cref{sec_furtherconvprop} establishes further convergence properties of Broyden's method like the violation of uniform linear independence. \Cref{sec_num} contains the numerical experiments and \Cref{sec_conclusion} provides a summary.

\emph{Notation} We use $\N_0:=\N\cup\{0\}$. By $\norm{\cdot}{}$ we denote the Euclidean norm and the spectral norm.

%%%%%%%%%%%%%%%%%%%%%%%%%%%%%%%%%%%%%%%%%%%%%%%%%%%%%%%%%%%%%%%%%%%%%%%%%%%%%%%%%
%%%%%%%%%%%%%%%%%%%%%%%%%%%%%%%%%%%%%%%%%%%%%%%%%%%%%%%%%%%%%%%%%%%%%%%%%%%%%%%%%

\section[Enlarging the domain of convergence]{Enlarging the domain of convergence for Broyden's method and some accelerated schemes}\label{sec_conviterates}

In this section we show that for the singular problems under consideration, the use of a preceding Newton--like step ensures convergence of Broyden's method and some accelerated schemes for all starting points in a starlike domain with density~1. 
This is a significant extension of the domain of convergence for these methods.
The following assumption describes the problem class under consideration. Here,  $\spann{\phi}$ denotes the linear hull of the vector $\phi$, $P_N\in\R^{n\times n}$ is the projection parallel to $X$ onto $N$, and $P_X=I-P_N$.

\begin{assumption}\label{A2}
	Let $F:\R^n\rightarrow\R^n$ be differentiable in a neighborhood of $\uopt$ and twice differentiable at $\uopt$, where $\uopt$ satisfies $F(\uopt)=0$. Moreover, suppose that the following conditions hold: 
	\begin{itemize}
		\item There is $\phi\in\R^n$ with $\norm{\phi}{2}=1$ such that 
		$N:=\ker(F'(\uopt))=\spann{\phi}$ and such that $X:=\bild(F'(\uopt))$ satisfies $X\oplus N=\R^n$. 
		\item There holds $P_N (F''(\uopt)(\phi,\phi))\neq 0$.
	\end{itemize}	
\end{assumption}

\begin{remark}	
	It follows from \cite[Theorem~2.1]{DeckerKelley_NewtonSingularJacobian} that under \Cref{A2} there is $\rho>0$ such that $\uopt$ is the only root of $F$ in $\Ballop_{\rho}(\uopt)$. That is, $\uopt$ is \emph{isolated}.
	Moreover, \Cref{A2} implies that $\uopt$ is a \emph{regular singularity of first order with a one-dimensional nullspace $N$}, cf. \cite[Section~2]{Griewank_StarlikedomainsofconvNewtSingSiamRev}---also called \emph{simple singularity} \cite[(2.7)]{Griewank_StarlikedomainsofconvNewtSingSiamRev}---as well as a
	\emph{strongly regular singularity} \cite[section~2, (45)]{Griewank_Diss}. Furthermore, if \Cref{A2} holds, then $F$ is \emph{2-regular at $\uopt$ in the direction $\phi$} \cite[Definition~2]{IzmailovKurennoySolodov_NewtonSingularJacobianLocalAttraction}.
	Finally, let us also mention that within the class of singular problems, first order singularities with 
	a one-dimensional nullspace are the ones occurring most frequently according to Griewank, cf. \cite[bottom of p.~540]{Griewank_StarlikedomainsofconvNewtSingSiamRev}. 
\end{remark}

\subsection{The result of Decker and Kelley for Broyden's method}

In \cite{DeckerKelley_BroydenSingularJacobian} it is shown that under \Cref{A2} there are constants $\rho>0$ and $\theta>0$ such that for every $u^0\in W_{\uopt}(\rho,\theta)$ and all $B_0$ close to $F'(u^0)$, Algorithm~\ref{alg_broy} is well-defined and generates a sequence $(u^k)$ that converges q-linearly to $\uopt$; for the definition of $W_{\uopt}(\rho,\theta)$, see \eqref{def_theta}. More precisely, the following result is available.

\begin{theorem}[See {\cite[Theorem~1.11]{DeckerKelley_BroydenSingularJacobian}}]\label{thm_fundlemsingularJacobian}
	Let \Cref{A2} hold, let $F:\R^n\rightarrow\R^n$ be twice differentiable in a neighborhood of $\uopt$ and three times differentiable at $\uopt$,
	and let $\mu_X,\mu_N>0$. Then there exist $\rho,\theta>0$ such that for all $(u^0,B_0)\in W_{\uopt}(\rho,\theta)\times\R^{n\times n}$ satisfying
	\begin{equation}\label{eq_errsB0}
		\left\{ \enspace 
		\begin{aligned}
			\norm{(B_0-F'(u^0)) P_X}{2} &\leq \mu_X\norm{u^0-\uopt}{2} \quad{and}\\
			\norm{(B_0-F'(u^0)) P_N}{2} &\leq \mu_N\norm{u^0-\uopt}{2}^2,
		\end{aligned}\right.
	\end{equation}
	Algorithm~\ref{alg_broy} is well-defined and generates a sequence $(u^k)\subset W_{\uopt}(\rho,\theta)$ such that 
	\begin{equation}\label{eq_conclestforsingJac2}
		\lim_{k\to\infty}\frac{\norm{u^{k+1}-\uopt}{2}}{\norm{u^k-\uopt}{2}} = \frac{\sqrt5 - 1}{2}
		\qquad\text{and}\qquad
		\lim_{k\to\infty}\frac{\norm{P_X(u^k-\uopt)}{2}}{\norm{P_N(u^k-\uopt)}{2}^2} = 0.
	\end{equation}
	Moreover, there holds for all $k\geq 1$ 
	\begin{equation}\label{eq_regbehavioronN}
		P_N(u^{k+1}-\uopt) = \lambda_k P_N(u^k-\uopt),
	\end{equation}	
	where $(\lambda_k)_{k\geq 1}$ satisfies $(\lambda_k)_{k\geq 1}\subset\left(\frac38,\frac45\right)$ and $\lim_{k\to\infty}\lambda_k = \frac{\sqrt5 - 1}{2}$.
\end{theorem}

\begin{proof}
	All claims are established in \cite{DeckerKelley_BroydenSingularJacobian}: 
	(1.13) in \cite{DeckerKelley_BroydenSingularJacobian} provides the left limit and from (2.13) together with $\theta_n\to 0$ we deduce the right limit.
	The claim \eqref{eq_regbehavioronN} is contained in \cite[(2.14)]{DeckerKelley_BroydenSingularJacobian}.
	(We note that the index ``$n+1$'' in \cite[(1.14)]{DeckerKelley_BroydenSingularJacobian} is a misprint that should read ``$n$''.)
\end{proof}

\begin{remark}
	\phantom{to enforce linebreak}
	\begin{enumerate}
		\item[1)] The conditions \eqref{eq_errsB0} correct an error in \cite[Theorem~1.11]{DeckerKelley_BroydenSingularJacobian}, where the right-hand sides contain $\rho$, respectively, $\rho^2$, instead of 
		$\norm{u^0-\uopt}{2}$ and $\norm{u^0-\uopt}{2}^2$, see \cite[(1.12)]{DeckerKelley_BroydenSingularJacobian}.
		The fact that the original $\rho$ and $\rho^2$ have to be replaced by 
		$\norm{u^0-\uopt}{2}$ and $\norm{u^0-\uopt}{2}^2$, respectively,
		is mandated by \cite[(2.4)]{DeckerKelley_BroydenSingularJacobian}, equation for $D_0$, where it is needed that, in their notation, $P_N E_0 P_N = P_N F'(u^0) P_N + O(\norm{u^0-\uopt}{2}^2)$ (our notation uses $B_0$ for $E_0$),
		and similarly for $A_0$ in \cite[(2.4)]{DeckerKelley_BroydenSingularJacobian}. 
		\item[2)] As pointed out by a referee, it would be interesting to see if a result similar to \Cref{thm_fundlemsingularJacobian} can be proven if \Cref{A2} is replaced by 2-regularity of $F$ at $\uopt$ in a direction of $N$. 
		This would, in particular, allow for non-isolated singularities, a setting that is excluded by \Cref{A2}. For Newton--type methods such a result is available in \cite[Theorem~1]{IzmailovKurennoySolodov_NewtonSingularJacobianLocalAttraction}. 
	\end{enumerate}
\end{remark}

\subsection{Enlarging the domain of convergence of Broyden's method by a preceding Newton--like step}

Using results of Griewank \cite{Griewank_StarlikedomainsofconvNewtSing,Griewank_StarlikedomainsofconvNewtSingSiamRev} and Decker and Kelley \cite{DeckerKelley_NewtonSingularJacobian} the domain of convergence $W_{\uopt}(\rho,\theta)$ in \Cref{thm_fundlemsingularJacobian} can be extended significantly by use of a preceding Newton--like step as described in Algorithm~\ref{alg_broypNl}. 
Specifically, we establish that it includes a starlike domain with density~1, see \Cref{sec_introenlarging} for definitions, while the convergence behavior of $(u^k)$ described by \Cref{thm_fundlemsingularJacobian} is maintained. 
This is one of the main results of this work. 

\begin{theorem}\label{thm_starlikedomainofconvforBroyden}
	Let \Cref{A2} hold and let $F$ be twice differentiable in a neighborhood of $\uopt$ and three times differentiable at $\uopt$.
	Then there exist constants $\rho,\theta>0$ and a starlike domain $\calW_{\uopt}$ with density~1 such that the following properties are satisfied:
	\begin{enumerate}
		\item[1)] For each $\hat u\in \calW_{\uopt}$ there is a number $\sigma=\sigma(\hat u)>0$ with the following properties: For every initial guess $(\hat u,\hat B)\in \calW_{\uopt}\times\Ballop_{\sigma}(F'(\hat u))$ in Algorithm~\ref{alg_broypNl}, $u^0$ is well-defined and satisfies $u^0\in W_{\uopt}\bigl(\frac34\rho,\frac34\theta\bigr)$. If, in addition, $B_0$ is selected such that $B_0\in\Ballop_{\sigma}(F'(u^0))$,
		then Algorithm~\ref{alg_broypNl} generates a sequence $(u^k)$ that satisfies \eqref{eq_conclestforsingJac2} and \eqref{eq_regbehavioronN}.
		\item[2)] There holds $W_{\uopt}(\rho,\theta)\subset\calW_{\uopt}$ and 
		every $(u^k)$ generated according to 1) satisfies $(u^k)\subset W_{\uopt}\bigl(\frac34\rho,\frac34\theta\bigr)$.
		\item[3)] For any compact set $W\subset\calW_{\uopt}$, the constant $\sigma$ in 1) can be chosen independently of $\hat u$ for all $\hat u \in W$.
	\end{enumerate}
\end{theorem}

\begin{proof}
	\textbf{Proof of 1) and 2):} To prove 1) we mainly have to show that for appropriately chosen $\sigma$, $\rho$ and $\theta$ the pair $(u^0,B_0)$ satisfies the assumptions of \Cref{thm_fundlemsingularJacobian}. In doing so we will also establish 2).
	We recall that the invertible matrices form an open set on which $B\mapsto B^{-1}$ is continuous. 
	
	Let $\mu_X:=\mu_N:=2$ and let $\rho_1,\theta_1>0$ be such that the assertions of \Cref{thm_fundlemsingularJacobian} are satisfied for every $(u^0,B_0)\in W_{\uopt}(\rho_1,\theta_1)$ such that $B_0$ satisfies \eqref{eq_errsB0}. Without loss of generality (wlog.) we can assume that $\rho_1\leq 1$, so $\min\{\rho_1,\rho_1^2\}=\rho_1^2$.
	From \cite[Theorem~1.2 and (2.7)--(2.8) in its proof]{DeckerKelley_NewtonSingularJacobian}
	we obtain $\rho_2\in(0,\rho_1)$ and $\theta_2\in(0,\theta_1)$ such that 
	for every $\hat u\in W_{\uopt}(\rho_2,\theta_2)$ the matrix $F'(\hat u)$ is invertible and the Newton iterate
	\begin{equation*}
		u^+:=\hat u - F'(\hat u)^{-1}F(\hat u)
	\end{equation*} 
	satisfies $u^+\in W_{\uopt}(\frac{11}{20}\rho_2,\frac{11}{20}\theta_2)$.  
	Thus, for every $\hat u\in W_{\uopt}(\rho_2,\theta_2)$ there is $\sigma_1\in(0,\rho_2)$ such that all $\hat B\in\Ballop_{\sigma_1}(F'(\hat u))$ are invertible, 
	\begin{equation*}
		u^0=\hat u-\hat B^{-1}F(\hat u)
	\end{equation*}
	satisfies $u^0\in W_{\uopt}(\frac{3}{5}\rho_2,\frac{3}{5}\theta_2)\subset W_{\uopt}(\rho_1,\theta_1)$, and all $B_0\in\mathbb{B}_{\sigma_1}(F'(u^0))$ are invertible and satisfy \eqref{eq_errsB0} for $\rho=\rho_1$ and $\theta=\theta_1$. 
	That is, $(u^0,B_0)$ satisfies the requirements of \Cref{thm_fundlemsingularJacobian} whenever $(\hat  u,\hat B)\in W_{\uopt}(\rho_2,\theta_2)\times\Ballop_{\sigma_1}(F'(\hat u))$. 	
	We next discuss the case $\hat u\notin W_{\uopt}(\rho_2,\theta_2)$.
	Let $\hat\calW_{\uopt}$ be the starlike domain with density $1$ from \cite[Theorem~6.1~(i)]{Griewank_StarlikedomainsofconvNewtSing} (denoted there by $\cal R$).
	It is shown in \cite[Theorem~6.1~(i)]{Griewank_StarlikedomainsofconvNewtSing} that Newton's method converges to $\uopt$ for all starting points in $\hat\calW_{\uopt}$. In particular, for any $\hat u\in\hat\calW_{\uopt}$ the matrix $F'(\hat u)$ is invertible, which implies that there is $\sigma_2\in(0,\sigma_1)$ such that 
	all $\hat B\in\Ballop_{\sigma_2}(F'(\hat u))$ are invertible. 
	The proof of \cite[Theorem~6.1~(i)]{Griewank_StarlikedomainsofconvNewtSing} shows that we can assume wlog. that for any $\hat u\in\hat\calW_{\uopt}$ the Newton iterate
	\begin{equation*}
		u^+:= \hat u - F'(\hat u)^{-1}F(\hat u)
	\end{equation*}
	satisfies $u^+\in W_{\uopt}(\frac{11}{20}\rho_2,\frac{11}{20}\theta_2)$
	(more precisely, this can be achieved by multiplying with sufficiently small constants the second and third term in the minimum that constitutes the definition of $r(t)$ in said proof).
	This implies that for each $\hat u\in\hat\calW_{\uopt}$ there is a constant $\sigma\in(0,\sigma_2)$ such that any $\hat B\in\Ballop_\sigma(F'(\hat u))$ is invertible, 
	\begin{equation*}
		u^0=\hat u - \hat B^{-1}F(\hat u)
	\end{equation*}
	satisfies $u^0\in W_{\uopt}(\frac35\rho_2,\frac35\theta_2)$,
	and all $B_0\in\mathbb{B}_\sigma(F'(u^0))$ are invertible and satisfy \eqref{eq_errsB0} for $\rho=\rho_1$ and $\theta=\theta_1$. 
	That is, $(u^0,B_0)$ satisfies the requirements of \Cref{thm_fundlemsingularJacobian} whenever $(\hat u,\hat B)\in \hat\calW_{\uopt}\times\Ballop_\sigma(F'(\hat u))$. 	
	
	With $\text{int}$ denoting the topological interior we find that $\calW_{\uopt}:=\hat\calW_{\uopt}\cup \text{int}(W_{\uopt}(\rho_2,\theta_2))$ is a starlike domain with density~$1$ that contains $W_{\uopt}(\frac45\rho_2,\frac45\theta_2)$ and satisfies all claims if we set $\rho:=\frac45\rho_2$ and $\theta:=\frac45\theta_2$.\\
	\textbf{Proof of 3):} A slight generalization of the arguments in the proof of 1) and 2) yields that for any
	$u\in\calW_{\uopt}$ there are $\sigma=\sigma(u)>0$ and $\hat\sigma=\hat\sigma(u)>0$ such that for every $\hat u\in\Ballop_{\hat\sigma}(u)$ all $\hat B\in\Ballop_{\sigma}(F'(\hat u))$ are invertible,
	\begin{equation*}
		u^0=\hat u - \hat B^{-1}F(\hat u)
	\end{equation*}
	satisfies $u^0\in W_{\uopt}(\frac35\rho_2,\frac35\theta_2)$,
	and all $B_0\in\mathbb{B}_{\sigma}(F'(u^0))$ are invertible and satisfy \eqref{eq_errsB0} for $\rho=\rho_1$ and $\theta=\theta_1$. 
	The open cover 
	\begin{equation*}
		W\subset\bigcup_{u\in W} \Ballop_{\hat\sigma(u)}(u)
	\end{equation*}
	of $W$ has a finite subcover, say $\cup_{i=1}^M \Ballop_{\hat\sigma(u_i)}(u_i)$; 
	the claim holds for $\sigma:=\min\{\sigma(u_i): i=1,\ldots,M\}$.
\end{proof}

\subsection{Enlarging the domain of convergence of accelerated schemes by a preceding Newton--like step}

The proof of \Cref{thm_starlikedomainofconvforBroyden} shows in particular that 
\begin{itemize}
	\item for all $\rho,\theta>0$ there is a starlike domain $\calW_{\uopt}$ with density~1 such that applying a Newton--like step to any $\hat u\in\calW_{\uopt}$ yields a point $u^0\in W_{\uopt}(\frac34\rho,\frac34\theta)$;
	\item if $W\subset\calW_{\uopt}$ is compact, the size of the perturbation $\hat B-F'(\hat u)$ can be taken independently of $\hat u\in W$ without violating $u^0=\hat u-\hat B^{-1}F(\hat u)\in W_{\uopt}(\frac34\rho,\frac34\theta)$.	
\end{itemize}
These results concern only the preceding Newton--like step in Algorithm~\ref{alg_broypNl}; they are independent of Broyden's method. 
Therefore, they still hold if \Cref{A2} is replaced by the weaker assumptions used in \cite{Griewank_StarlikedomainsofconvNewtSing}; 
\Cref{A2} is required in \Cref{thm_starlikedomainofconvforBroyden} only to ensure that Broyden's method in Algorithm~\ref{alg_broypNl} is well-behaved.
Since the results on the preceding Newton--like step are not related to Broyden's method specifically, they can be applied to other methods as well, producing statements similar to \Cref{thm_starlikedomainofconvforBroyden}. For instance, we can apply them to finite-dimensional versions of the acceleration schemes presented in \cite[Theorem~1.3]{KelleySuresh_AccelerationNewtonSingularJacobian},  
\cite[Theorem~1.4]{KelleySuresh_AccelerationNewtonSingularJacobian}, and \cite[Theorem~1.4]{Kelley_ShamanskiiLikeAcceleration}.
We spell this out for the Shamanskii--like method of \cite[Theorem~1.4]{Kelley_ShamanskiiLikeAcceleration}, which is a three-point scheme combined with an initial Newton step. Kelley's algorithm together with the preceding Newton--like step that we propose is provided in Algorithm~\ref{alg_Shamlike}. 
To comply with the notation of \cite[Theorem~1.4]{Kelley_ShamanskiiLikeAcceleration}, the iterate resulting from the preceding Newton--like step is denoted $u^{-1}$.

\begin{algorithm2e}
	\SetAlgoRefName{SMP}
	\DontPrintSemicolon
	\caption{Kelley's Shamanskii--like method with a preceding Newton--like step}
	\label{alg_Shamlike}
	\KwIn{ $\hat u\in\R^n$, \, $\hat B\in\R^{n\times n}$ \text{invertible}, \, 
		$C\in\R\setminus\{0\}$, \, $\alpha\in(0,\frac{\sqrt5 - 1}{2})$
	}
	Let $u^{-1}:=\hat u - \hat B^{-1} F(\hat u)$
	\tcp*{preceding Newton-like step}
	Let $u^0:=u^{-1} - F'(u^{-1})^{-1} F(u^{-1})$
	\tcp*{begin of Kelley's algorithm}
	\For{ $k=0,1,2,\ldots$ }
	{   
		\lIf{$F(u^k)=0$ }{let $u^\ast:=u^k$; STOP}
		Let $y^{k+1}:=u^k-F'(u^k)^{-1} F(u^k)$
		\tcp*{Newton iterate}
		Let $z^{k+1}:= F'(u^k)^{-1} F(y^{k+1})$
		\tcp*{simplified Newton step}
		Let $u^{k+1}:=y^{k+1}-\left( 4 - C\norm{z^{k+1}}{2}^\alpha\right) z^{k+1}$
		\tcp*{correction of Newton iterate}
	}
	\KwOut{ $u^\ast$ }
\end{algorithm2e}

We note that Broyden steps are not involved in Algorithm~\ref{alg_Shamlike}. In fact, there seem to be no accelerated methods available that involve Broyden steps (except for the remarks in \cite[Section~5]{Griewank_StarlikedomainsofconvNewtSingSiamRev}).
The advantage of accelerated methods such as Algorithm~\ref{alg_Shamlike} is that their convergence is of higher order in comparison to Newton--type methods and Broyden's method. This comes at the cost of stronger differentiability requirements, cf. Kelley's result \cite[Theorem~1.4]{Kelley_ShamanskiiLikeAcceleration}.
Because we want to invoke Kelley's result, we have to require the same amount of differentiability. 

\begin{theorem}\label{thm_starlikedomainofconvforShamanskii}
	Let \Cref{A2} hold and let $F$ be four times continuously differentiable in a neighborhood of $\uopt$. 
	Then there exist constants $\rho,\theta>0$ and a starlike domain $\calW_{\uopt}$ with density~1 such that the following properties are satisfied:
	\begin{enumerate}
		\item[1)] For every $\hat u\in \calW_{\uopt}$ there is a number $\sigma=\sigma(\hat u)>0$ such that for each $(\hat u,\hat B)\in \calW_{\uopt}\times\Ballop_{\sigma}(F'(\hat u))$ the iterates
		$(u^k)_{k\geq -1}$ of Algorithm~\ref{alg_Shamlike} are well-defined
		and satisfy $(u^k)_{k\geq 0}\subset W_{\uopt}(\frac34\rho,\frac34\theta)$. Moreover, $(u^k)$ converges to $\uopt$ with q-order $1+\alpha$, i.e.,
		\begin{equation*}
			\limsup_{k\to\infty} \, \frac{\norm{u^{k+1}-\uopt}{}}{\norm{u^k-\uopt}{2}^{1+\alpha}} < \infty.
		\end{equation*}
		\item[2)] For any compact set $W\subset\calW_{\uopt}$, the constant $\sigma$ in 1) can be chosen independently of $\hat u$ for all $\hat u \in W$.
	\end{enumerate}
\end{theorem}

\begin{proof}
	It is not difficult to check that the prerequisites of
	\cite[Theorem~1.4]{Kelley_ShamanskiiLikeAcceleration} are satisfied. To simplify that check for the reader, let us mention three points:
	\begin{itemize}
		\item The integer $k$ appearing in \cite[Theorem~1.4]{Kelley_ShamanskiiLikeAcceleration} is the order of the singularity, i.e., $k=1$.
		\item The set $W(\rho,\theta,\eta)$ appearing in \cite[Theorem~1.4]{Kelley_ShamanskiiLikeAcceleration} is identical to our $W(\rho,\theta)$ since, in the notation of \cite[Theorem~1.4]{Kelley_ShamanskiiLikeAcceleration}, $\delta P_N \equiv 0$ due to $k=1$.
		\item In the setting under consideration, the two conditions on $P_N F''(\uopt)$ appearing in the assumption \cite[(R)]{Kelley_ShamanskiiLikeAcceleration} are identical to each other and agree with our condition $P_N (F''(\uopt)(\phi,\phi))\neq 0$ imposed as part of \Cref{A2}. \qedhere
	\end{itemize}
\end{proof}

%%%%%%%%%%%%%%%%%%%%%%%%%%%%%%%%%%%%%%%%%%%%%%%%%%%%%%%%%%%%%%%%%%%%%%%%%%%%%%%%%
%%%%%%%%%%%%%%%%%%%%%%%%%%%%%%%%%%%%%%%%%%%%%%%%%%%%%%%%%%%%%%%%%%%%%%%%%%%%%%%%%

\section{Convergence of the Broyden updates and matrices}\label{sec_main}

In this section we establish convergence of the Broyden matrices for the class of singular problems described by \Cref{A2}.
Specifically, in \Cref{sec_mainr} we show that the Broyden updates converge r-linearly to zero under certain assumptions. Under stronger assumptions we prove in \Cref{sec_mainq} that their convergence is q-linear. Both results imply that the Broyden matrices converge.
The assumptions of \Cref{sec_mainq} are in particular satisfied if the iterates are generated by Algorithm~\ref{alg_broypNl} as described in \Cref{thm_starlikedomainofconvforBroyden}. Thus, the Broyden updates converge q-linearly for all starting points in a domain with density~1 if a preceding Newton step is used and both $\norm{\hat B-F'(\hat u)}{}$ and $\norm{B_0-F'(u^0)}{}$ are sufficiently small. In the numerical experiments we have sometimes  observed that the Broyden iterates and updates converge only r-linearly, and this type of convergence is covered by the results from \Cref{sec_mainr}.

Let us introduce additional notation.
If Algorithm~\ref{alg_broy} generates a sequence $(u^k)$, then necessarily $s^k\neq 0$ for all $k$; we will use this tacitly. 
Let $F:\R^n\rightarrow\R^n$ and let $(u^k)$, $(s^k)$ and $(B_k)$ be generated by Algorithm~\ref{alg_broy}. 
For $k\in\N_0$ we set
\begin{equation*}
	\hat s^k := 
	\frac{s^k}{\norm{s^k}{2}}
	\quad\qquad\text{ and }\qquad\quad
	\varepsilon_k := \frac{\norm{F(u^{k+1})}{2}}{\norm{s^k}{2}}.
\end{equation*}
Since $B_k s^k=-F(u^k)$ and since $\lVert v w^T \rVert = \lVert v\rVert$ for any $v,w\in\R^n$ with $\lVert w\rVert =1$, we observe that $\varepsilon_k=\frac{\norm{F(u^{k+1})-F(u^k)-B_k s^k}{2}}{\norm{s^k}{2}}=\norm{B_{k+1}-B_k}{2}$, i.e., $\varepsilon_k$ is the norm of the Broyden update.
If, in addition, the limit $\uopt:=\lim_{k\to\infty} u^k$ exists and $F$ is differentiable at $\uopt$, then we set 
\begin{equation*}
	E_k:= B_k - F'(\uopt), \quad k\in\N_0.
\end{equation*}

\subsection{R-linear convergence of the Broyden updates}\label{sec_mainr}

We first establish sufficient conditions for r-linear convergence of the Broyden updates to zero. 
Then we prove some convergence properties of Broyden's method which allow us to conclude that the sufficient conditions are satisfied 
if $(u^k)$ is generated according to \Cref{thm_fundlemsingularJacobian} or \Cref{thm_starlikedomainofconvforBroyden}.

\subsubsection{The sufficient conditions}

The sufficient conditions will follow from the next two lemmas.  

\begin{lemma}\label{lem_convcondbasedonepsilons}
	Let $(\beta_k)\subset[0,\infty)$ and suppose there are constants $C,\gamma>0$ and a sequence $(\alpha_k)\subset(0,\infty)$ 
	such that 
	\begin{equation*}
		\limsup_{k\to\infty}\frac{\beta_k}{\alpha_k^\gamma}\leq C \qquad\text{ and }\qquad
		\limsup_{k\to\infty}\sqrt[k]{\alpha_k} \leq\bar\kappa
	\end{equation*}
	hold for some $\bar\kappa\in[0,\infty)$. Then $\limsup_{k\to\infty}\sqrt[k]{\beta_k}\leq\bar\kappa^\gamma$. 
\end{lemma}

\begin{proof}
	The assumptions imply that for any $\delta>0$ there is $K\in\N_0$ such that 
	$\sqrt[k]{\beta_k}\leq \sqrt[k]{2C}\sqrt[k]{\bar\kappa+\delta}^\gamma$ holds for all $k\geq K$. 
	This yields the claim. 
\end{proof}

\begin{lemma}\label{lem_convlem2}
	Let $(v^k)\subset\R^n$ be a sequence satisfying $\limsup_{k\to\infty}\sqrt[k]{\norm{v^k}{}} \leq\bar\kappa$
	for some $\bar\kappa\in[0,\infty)$. Then $\limsup_{k\to\infty}\sqrt[k]{\norm{v^{k+1}-v^k}{}}\leq\bar\kappa$. 
\end{lemma}

\begin{proof}
	From
	$\limsup_{k\to\infty}\sqrt[k]{\norm{v^k}{}} \leq\bar\kappa$ we infer 
	$\limsup_{k\to\infty}\sqrt[k]{\norm{v^{k+1}}{}} \leq\bar\kappa$. The claim thus follows by use of 
	$\norm{v^{k+1}-v^k}{} \leq 2\max\{\norm{v^{k+1}}{},\norm{v^k}{}\}$. 
\end{proof}

We obtain sufficient conditions for r-linear convergence of $(\varepsilon_k)$ to zero. 

\begin{corollary}\label{cor_rlinearrateforupdates}
	Let $F:\R^n\rightarrow\R^n$ and $\uopt\in\R^n$. Let $(u^k)$ be generated by Algorithm~\ref{alg_broy} and suppose that there are $C,\gamma>0$ and $\kappa\in[0,1)$ such that
	\begin{equation*}
		\limsup_{k\to\infty} \, \sqrt[k]{\norm{u^{k}-\uopt}{2}}\leq \kappa
		\qquad\text{ and }\qquad
		\limsup_{k\to\infty} \, \frac{\norm{F(u^{k+1})}{2}}{\norm{s^k}{2}^{1+\gamma}} \leq C.
	\end{equation*}
	Then $\limsup_{k\to\infty}\sqrt[k]{\varepsilon_k}\leq\kappa^\gamma$ and $\sum_k\varepsilon_k<\infty$.
\end{corollary}

\begin{proof}
	The claim $\sum_k\varepsilon_k<\infty$ follows from $\limsup_{k\to\infty}\sqrt[k]{\varepsilon_k}\leq\kappa^\gamma<1$. To establish the latter, we apply 
	\Cref{lem_convlem2} to $(v^k)\equiv (u^k-\uopt)$. This yields $\limsup_{k\to\infty}\sqrt[k]{\norm{s^k}{}}\leq\kappa$.
	The claim thus follows from \Cref{lem_convcondbasedonepsilons}, applied to
	$(\alpha_k)\equiv(\norm{s^k}{})$ and $(\beta_k)\equiv (\varepsilon_k)$. 
\end{proof}	

\begin{remark}
	\Cref{cor_rlinearrateforupdates} provides sufficient conditions for 
	$\sum_k\varepsilon_k<\infty$ to hold, which due to $\varepsilon_{k}=\norm{B_{k+1}-B_k}{2}$ are also sufficient conditions for convergence of 
	$(B_k)$. Since the available results for convergence of $(B_k)$ are mostly based on uniform linear independence of $(\hat s^k)$, see the discussion in \Cref{sec_introconvBrmatrices}, we stress that \Cref{cor_rlinearrateforupdates} does not rely on uniform linear independence of $(\hat s^k)$.
	In fact, we will use it to demonstrate convergence of $(B_k)$ in a setting where uniform linear independence is provably violated.
	Note that the existence of $F'(\uopt)$ is also not required, let alone its regularity or singularity.
	Finally, we observe that \Cref{cor_rlinearrateforupdates} implies $u^k\to \uopt$, $s^k\to 0$ and $F(u^k)\to 0$ for $k\to\infty$.
\end{remark}

\subsubsection{Convergence properties of Broyden's method}

To show that the sufficient conditions for r-linear convergence of the Broyden updates are satisfied, we establish the following convergence theorem that enriches \cite{DeckerKelley_BroydenSingularJacobian} with additional convergence properties of $(u^k)$, $(s^k)$ and $(F(u^k))$. 

\begin{theorem}\label{thm_convsingularcase}
	Let \Cref{A2} hold. Let $(u^k)$ be generated by Algorithm~\ref{alg_broy} and suppose that it satisfies
	\begin{equation}\label{eq_conclestforsingJac}
		\limsup_{k\to\infty}\frac{\norm{u^{k+1}-\uopt}{2}}{\norm{u^k-\uopt}{2}} \leq \kappa 
		\qquad\text{ and }\qquad
		\lim_{k\to\infty}\frac{\norm{P_X(u^k-\uopt)}{2}}{\norm{P_N(u^k-\uopt)}{2}^2} = 0
	\end{equation}
	for some $\kappa\in[0,1)$.
	Then:
	\begin{enumerate}
		\item[1)] There holds $\min\{\norm{\hat s^k - \phi}{2},\norm{\hat s^k + \phi}{2}\}=o(\norm{u^k-\uopt}{2})$ for $k\to\infty$. 
		\item[2)] There holds 
		\begin{equation*}
			\lim_{k\to\infty}\frac{\norm{F(u^k)}{2}}{\norm{u^k-\uopt}{2}^2}
			= \norm{F''(\uopt)(\phi,\phi)}{2}.
		\end{equation*}	
		\item[3)] There exists a constant $C>0$ such that for all $k\in\N_0$
		\begin{equation*}
			\norm{F(u^{k+1})}{2} \leq C\norm{s^k}{2}^2.
		\end{equation*} 
		In particular, $\delta_k:=\frac{\log(\norm{F(u^k)}{2})}{\log(\norm{s^{k-1}}{2})}$, $k\in\N$, satisfies $\liminf_{k\to\infty}\delta_k\geq 2$. 
	\end{enumerate}
	In particular, the statements 1)--3) hold if $(u^k)$ is generated according to \Cref{thm_fundlemsingularJacobian} or \Cref{thm_starlikedomainofconvforBroyden}.
\end{theorem}

\begin{proof}
	We will use tacitly that the norms of $P_X$ and $P_N$ are bounded by $1$ and that $P_X(v)+P_N(v)=v$ for all $v\in\R^n$.
	In addition, we can assume without loss of generality that $X=N^\perp$, see \cite[beginning of the proof of Theorem~1.11]{DeckerKelley_BroydenSingularJacobian}, so 
	$P_N(v)=v^T\phi \phi$ for all $v\in\R^n$.\\
	\textbf{Proof of 1):} 
	For all $k\in\N_0$ there holds
	\begin{equation*}
		\norm{P_N(\hat s^k)\pm\phi}{2}
		= \norm{(\hat s^k)^T \phi\phi \pm \phi}{2}
		= 1 \pm (\hat s^k)^T\phi
		= \frac12\norm{\hat s^k\pm\phi}{2}^2,
	\end{equation*}
	which implies
	\begin{equation*}
		\norm{\hat s^k\pm\phi}{2}
		= \norm{P_X(\hat s^k) + P_N(\hat s^k) \pm\phi}{2}
		\leq 
		\norm{P_X(\hat s^k)}{2} + \frac12\norm{\hat s^k \pm\phi}{2}^2,
	\end{equation*}
	hence
	\begin{equation*}
		\left(1-\frac12\norm{\hat s^k\pm\phi}{2}\right)\norm{\hat s^k\pm\phi}{2}
		\leq \norm{P_X(\hat s^k)}{2}.
	\end{equation*}
	As $\min\{\norm{a+b}{2},\norm{a-b}{2}\} \leq \sqrt2$ for all $a,b\in\R^n$ with $\norm{a}{2}=\norm{b}{2}=1$, 
	we infer that
	\begin{equation*}
		\left(1-\frac{\sqrt2}{2}\right)\min\Bigl\{\norm{\hat s^k-\phi}{2},\norm{\hat s^k+\phi}{2}\Bigr\}
		\leq \norm{P_X(\hat s^k)}{2}.
	\end{equation*}
	It is therefore enough to show
	$\norm{P_X (\hat s^k)}{2}=o(\norm{u^k-\uopt}{2})$ for $k\to\infty$. 
	Denoting
	\begin{equation}\label{def_omegak}
		\omega_k:=\frac{\norm{P_X(u^k-\uopt)}{2}}{\norm{P_N(u^k-\uopt)}{2}^2}, \quad k\in\N_0, 
	\end{equation} 
	we have $\lim_{k\to\infty}\omega_k=0$ by \eqref{eq_conclestforsingJac}. Moreover, for all $k$ large enough there holds 
	\begin{equation*}
		\begin{split}
			\norm{P_X(s^k)}{2}
			& = \norm{P_X (u^{k+1}-u^k)}{2}
			\leq \norm{P_X (u^{k+1}-\uopt)}{2}
			+ \norm{P_X (u^k-\uopt)}{2}\\
			& = \omega_{k+1}\norm{P_N (u^{k+1}-\uopt)}{2}^2+\omega_k\norm{P_N (u^k-\uopt)}{2}^2
			\leq \left(\omega_{k+1}+\omega_k\right) \norm{u^k-\uopt}{2}^2,
		\end{split}
	\end{equation*}
	where we used the q-linear convergence of $(u^k)$ to obtain the last inequality.
	For sufficiently large $k$ we thus have 
	\begin{equation}\label{eq_auxsingjac}
		\norm{P_X (\hat s^k)}{2}
		= \frac{\norm{P_X(s^k)}{2}}{\norm{s^k}{2}}
		\leq \frac{\omega_{k+1}+\omega_k}{1-\kappa_\delta}\norm{u^k-\uopt}{2},
	\end{equation}
	where $\kappa_\delta:=\kappa+\delta$ for an arbitrary $\delta\in(0,1-\kappa)$ and we used the estimate $\norm{u^k-\uopt}{2}\leq \frac{1}{1-\kappa_\delta}\norm{s^k}{2}$ that holds for all sufficiently large $k$ due to
	\begin{equation}\label{eq_auxjacosing}
		\left(1-\kappa_\delta\right)\norm{u^k-\uopt}{2}
		\leq \norm{u^k-\uopt}{2} - \norm{u^{k+1}-\uopt}{2} 
		\leq \norm{u^{k+1}-u^k}{2} 
		= \norm{s^k}{2}.
	\end{equation}
	In view of $\lim_{k\to\infty}\omega_k=0$, \eqref{eq_auxsingjac} concludes the proof of 1).\\	
	\textbf{Proof of 2):} 
	For $k\in\N_0$ we set $\tilde u^{k+1}:=u^{k+1}-\uopt$. 
	Using $F(\uopt)=0$ we have 
	\begin{equation}\label{eq_TEsuffJacob}
		F(u^{k+1}) = F'(\uopt)(\tilde u^{k+1}) + F''(\uopt)(\tilde u^{k+1},\tilde u^{k+1}) + o\bigl(\norm{\tilde u^{k+1}}{2}^2\bigr).
	\end{equation}
	Since $F'(\uopt)(P_N(v)) = 0$ for all $v\in\R^n$ due to $N=\ker(F'(\uopt))$, \eqref{def_omegak} yields 
	\begin{equation*}
		\norm{F'(\uopt)(\tilde u^{k+1})}{2}
		= \norm{F'(\uopt)(P_X(\tilde u^{k+1}))}{2}
		\leq \norm{F'(\uopt)}{2}\omega_{k+1}\norm{\tilde u^{k+1}}{2}^2 = o\bigl(\norm{\tilde u^{k+1}}{2}^2\bigr)
	\end{equation*}
	for $k\to\infty$. 
	In view of \eqref{eq_TEsuffJacob} this shows that it is enough to demonstrate 
	\begin{equation*}
		\lim_{k\to\infty} \, \Bigl\lvert\norm{F''(\uopt)(\hat u^{k+1},\hat u^{k+1})}{2}
		- \norm{F''(\uopt)(\phi,\phi)}{2}\Bigr\rvert = 0,
	\end{equation*}
	where $\hat u^{k+1}:=\frac{\tilde u^{k+1}}{\norm{\tilde u^{k+1}}{2}}$. For this limit to vanish it suffices that 
	\begin{equation}\label{eq_finconcljs}
		\lim_{k\to\infty} \min\Bigl\{\norm{\hat u^{k+1}-\phi}{2},\norm{\hat u^{k+1}+\phi}{2}\Bigr\} = 0.
	\end{equation} 
	By arguments similar to those in the proof of 1) we derive
	\begin{equation*}
		\min\Bigl\{\norm{\hat u^{k+1}-\phi}{2},\norm{\hat u^{k+1}+\phi}{2}\Bigr\}
		= o\bigl(\norm{\tilde u^{k+1}}{2}\bigr)
	\end{equation*}
	for $k\to\infty$, which implies \eqref{eq_finconcljs}.\\
	\textbf{Proof of 3):} 
	Since $\norm{u^{k+1}-\uopt}{2}^2 \leq (\frac{\kappa_\delta}{1-\kappa_\delta}\norm{s^k}{2})^2$ by \eqref{eq_auxjacosing} for all $k$ large enough, the first inequality follows from part~2). 
	The claim concerning $(\delta_k)$ is obvious. 
	\\\textbf{Proof of the ``In particular, ...'' claim:} 
	If $(u^k)$ is generated according to \Cref{thm_fundlemsingularJacobian} or \Cref{thm_starlikedomainofconvforBroyden}, then 
	$(u^k)$ is generated by Algorithm~\ref{alg_broy} and satisfies \eqref{eq_conclestforsingJac2}, hence 
	\eqref{eq_conclestforsingJac} holds for $\kappa=(\sqrt5 -1)/2\in(0,1)$.
	As we have already established, this implies 1)--3).
\end{proof}	

\begin{remark}\label{rem_errorboundcond}
	Broyden's method and Newton--type methods can be studied under a so-called \emph{local error bound condition}, cf. e.g.
	\cite{Griewank_87,BroydensMethodInHilbertSpacesSetValued_ArtachoBelyakovDontchevLopez,DontRock_InexactNewton}.
	For single-valued mappings, this condition requires the existence of $L\geq 0$ such that 
	$\norm{u-v}{2}\leq L\norm{F(u)-F(v)}{2}$ for all $u,v$ near the root $\uopt$ and is therefore equivalent to \emph{metric regularity} \cite{DontRock,Ioffe_MetricRegularitySurvey}.
	The special case $v=\uopt$ also appears frequently and corresponds to \emph{metric subregularity}. While invertibility of $F'(\uopt)$ implies metric subregularity, 
	it can also hold if $F'(\uopt)$ is singular or does not exist. 
	The singular problems considered in this work are, however, not covered by metric subregularity. Indeed, by \Cref{A2} we have
	$F''(\uopt)(\phi,\phi)\neq 0$, hence part~2) of \Cref{thm_convsingularcase} readily implies that $L_k:=\norm{u^k-\uopt}{2}/\norm{F(u^k)-F(\uopt)}{2}$ goes to infinity.
\end{remark}	

\subsubsection{The main result}

As one of the main results of this work we show that the Broyden updates converge r-linearly to zero if $(u^k)$ is generated according to \Cref{thm_fundlemsingularJacobian} or \Cref{thm_starlikedomainofconvforBroyden}.

\begin{theorem}\label{thm_rlinconvofBrupdates}
	Let $F:\R^n\rightarrow\R^n$ and $\uopt\in\R^n$. Let $(u^k)$ be generated by Algorithm~\ref{alg_broy} and suppose that there are constants $C>0$ and  $\kappa\in[0,1)$ such that
	\begin{equation}\label{eq_rlinconvassumpineq}
		\limsup_{k\to\infty} \, \sqrt[k]{\norm{u^k-\uopt}{2}}\leq \kappa
		\qquad\text{ and }\qquad
		\limsup_{k\to\infty} \, \frac{\norm{F(u^{k+1})}{2}}{\norm{s^k}{2}^2} \leq C.
	\end{equation} 
	Then
	\begin{equation}\label{eq_rlinconvBroyupd}
		\limsup_{k\to\infty} \, \sqrt[k]{\varepsilon_k}\leq\kappa
		\qquad\text{ and }\qquad 
		\sum_k\varepsilon_k<\infty.
	\end{equation}
	In particular, \eqref{eq_rlinconvBroyupd} holds for $\kappa=\frac{\sqrt5 - 1}{2}$ if $(u^k)$ is generated according to \Cref{thm_fundlemsingularJacobian} or \Cref{thm_starlikedomainofconvforBroyden}.
\end{theorem}

\begin{proof}
	The inequalities \eqref{eq_rlinconvBroyupd} follow from \Cref{cor_rlinearrateforupdates} (taking $\gamma=1$ in that result).
	If $(u^k)$ is generated according to \Cref{thm_fundlemsingularJacobian} or \Cref{thm_starlikedomainofconvforBroyden}, then 
	$(u^k)$ is generated by Algorithm~\ref{alg_broy} and 
	satisfies \eqref{eq_conclestforsingJac2}, which implies that the first inequality in \eqref{eq_rlinconvassumpineq} holds for $\kappa=(\sqrt5 -1)/2\in(0,1)$. 
	Since \Cref{thm_convsingularcase}~3) holds as well, we infer that the second inequality in \eqref{eq_rlinconvassumpineq} is also satisfied. 
	As we have already established, this implies \eqref{eq_rlinconvBroyupd}. 
\end{proof}

\begin{remark}
	In \Cref{thm_qlinearrateforupdates} we improve the r-linear convergence of $(\varepsilon_k)$ to q-linear convergence. 
	We stress that \Cref{thm_rlinconvofBrupdates} requires significantly weaker assumptions than \Cref{thm_qlinearrateforupdates}.
\end{remark}

\subsection{Q-linear convergence of the Broyden updates}\label{sec_mainq}

To show that the Broyden updates converge q-linearly to zero, we first
establish sufficient conditions for this q-linear convergence. 
Then we demonstrate that these sufficient conditions are satisfied if $(u^k)$ is generated according to \Cref{thm_fundlemsingularJacobian} or \Cref{thm_starlikedomainofconvforBroyden}.

\subsubsection{The sufficient conditions}

We prove that if $(u^k)$ converges q-linearly with an \emph{asymptotic} q-factor and satisfies two additional assumptions, then the Broyden updates converge q-linearly with the same asymptotic q-factor. 

\begin{lemma}\label{lem_qlinearrateofupdates}
	Let \Cref{A2} hold and let $(u^k)$ be generated by Algorithm~\ref{alg_broy}. Suppose that 
	\begin{equation}\label{eq_assthm32}
		\lim_{k\to\infty}\frac{\norm{u^{k+1}-\uopt}{2}}{\norm{u^k-\uopt}{2}} = \kappa
		\qquad\text{ and }\qquad
		\lim_{k\to\infty}\frac{\norm{P_X(u^k-\uopt)}{2}}{\norm{P_N(u^k-\uopt)}{2}^2} = 0
	\end{equation}
	for some $\kappa\in(0,1)$. Moreover, suppose that 
	\begin{equation}\label{eq_limexpos}
		\lim_{k\to\infty}\frac{\norm{s^k}{2}}{\norm{u^k-\uopt}{2}} 
	\end{equation}
	exists and is positive. 
	Then
	\begin{equation*}
		\lim_{k\to\infty}\frac{\varepsilon_{k+1}}{\varepsilon_k} = \kappa.
	\end{equation*}
\end{lemma}	

\begin{proof}
	From part~2) of \Cref{thm_convsingularcase} and \eqref{eq_assthm32} it follows that
	\begin{equation*}
		\lim_{k\to\infty}\frac{\norm{F(u^{k+2})}{2}}{\norm{F(u^{k+1})}{2}} = 
		\lim_{k\to\infty}\frac{\norm{u^{k+2}-\uopt}{2}^2}{\norm{u^{k+1}-\uopt}{2}^2} = \kappa^2.
	\end{equation*}
	By assumption the limit $\lim_{k\to\infty}\frac{\norm{s^k}{2}}{\norm{u^k-\uopt}{2}}$ exists and is positive, hence
	\begin{equation*}
		\lim_{k\to\infty}\frac{\norm{s^k}{2}}{\norm{s^{k+1}}{2}} = 
		\lim_{k\to\infty}\frac{\kappa\norm{u^k-\uopt}{2}}{\kappa\norm{u^{k+1}-\uopt}{2}} = \frac{1}{\kappa}.
	\end{equation*}
	The claim thus follows from
	\begin{equation*}
		\frac{\varepsilon_{k+1}}{\varepsilon_k} = 
		\frac{\norm{F(u^{k+2})}{2}}{\norm{F(u^{k+1})}{2}} \cdot \frac{\norm{s^k}{2}}{\norm{s^{k+1}}{2}}, 
	\end{equation*}
	which holds by definition.
\end{proof}	

\subsubsection{The main result}

As one of the main results of this work we show that the Broyden updates converge q-linearly to zero if $(u^k)$ is generated according to \Cref{thm_fundlemsingularJacobian} or \Cref{thm_starlikedomainofconvforBroyden}.

\begin{theorem}\label{thm_qlinearrateforupdates}
	Let \Cref{A2} hold and let $(u^k)$ be generated by Algorithm~\ref{alg_broy}. Suppose that
	\begin{equation}\label{eq_convassqlinconvofupd}
		\lim_{k\to\infty}\frac{\norm{u^{k+1}-\uopt}{2}}{\norm{u^k-\uopt}{2}} = \frac{\sqrt5 - 1}{2} 
		\qquad\text{ and }\qquad
		\lim_{k\to\infty}\,\frac{\norm{P_X(u^k-\uopt)}{2}}{\norm{P_N(u^k-\uopt)}{2}^2} = 0
	\end{equation}
	as well as \eqref{eq_regbehavioronN} are satisfied.
	Then 
	\begin{equation}\label{eq_qlinconvofBrupd}
		\lim_{k\to\infty}\frac{\varepsilon_{k+1}}{\varepsilon_k} = \frac{\sqrt5 - 1}{2}.
	\end{equation}
	In particular, \eqref{eq_qlinconvofBrupd} holds if $(u^k)$ is generated according to \Cref{thm_fundlemsingularJacobian} or \Cref{thm_starlikedomainofconvforBroyden}.
\end{theorem}

\begin{proof}
	Regarding the main assertion we note that due to \Cref{lem_qlinearrateofupdates}, applied with $\kappa=(\sqrt5 -1)/2$, it suffices to show existence and positivity of the limit 
	$\hat\kappa:=\lim_{k\to\infty}\frac{\norm{s^k}{2}}{\norm{u^k-\uopt}{2}}$.
	We prove that $\hat\kappa=\frac{3-\sqrt5}{2}$.
	By taking the limit $k\to\infty$ in 
	\begin{equation*}
		\frac{\norm{s^k}{2}}{\norm{u^k-\uopt}{2}} \geq 
		\frac{\norm{u^k-\uopt}{2}-\norm{u^{k+1}-\uopt}{2}}{\norm{u^k-\uopt}{2}}
		\geq 1 - \frac{\norm{u^{k+1}-\uopt}{2}}{\norm{u^k-\uopt}{2}}
	\end{equation*}
	it is enough to demonstrate for all sufficiently large $k$ that
	\begin{equation}\label{eq_conclineq}
		\frac{\norm{s^k}{2}}{\norm{u^k-\uopt}{2}} \leq 1-\tilde\lambda_k,
	\end{equation}
	where $\lim_{k\to\infty}\tilde\lambda_k=\frac{\sqrt5 - 1}{2}$.
	As in previous proofs we derive for all large $k$ that
	\begin{equation}\label{eq_estimatelong}
		\begin{split}
			\norm{s^k}{2} & = \norm{P_X(s^k) + P_N(u^{k+1}-\uopt) - P_N(u^k-\uopt)}{2}\\
			& \leq \omega_{k+1} \norm{u^{k+1}-\uopt}{2}^2 + \omega_k\norm{u^k-\uopt}{2}^2 + (1-\lambda_k)\norm{P_N(u^k-\uopt)}{2}\\
			& \leq \left(\omega_{k+1} +\omega_k\right) \norm{u^k-\uopt}{2}^2+ (1-\lambda_k)\norm{u^k-\uopt}{2},
		\end{split}
	\end{equation}
	where $(\omega_k)$ is given in \eqref{def_omegak} and
	we used that due to \eqref{eq_regbehavioronN} there holds $P_N(u^{k+1}-\uopt)=\lambda_k P_N (u^k-\uopt)$ for all $k\geq 1$ with a sequence $(\lambda_k)\subset(\frac38,\frac45)$ that satisfies $\lim_{k\to\infty}\lambda_k=\frac{\sqrt5 -1}{2}$.
	Since $\omega_k\to 0$ for $k\to\infty$, \eqref{eq_estimatelong} implies \eqref{eq_conclineq}, which concludes the proof of the main assertion. 
	
	If $(u^k)$ is generated according to \Cref{thm_fundlemsingularJacobian} or \Cref{thm_starlikedomainofconvforBroyden}, then 
	$(u^k)$ is generated by Algorithm~\ref{alg_broy} and 
	satisfies \eqref{eq_conclestforsingJac2}, which is identical to \eqref{eq_convassqlinconvofupd}, as well as \eqref{eq_regbehavioronN}. 
	As we have already established, this implies \eqref{eq_qlinconvofBrupd}. 
\end{proof}

%%%%%%%%%%%%%%%%%%%%%%%%%%%%%%%%%%%%%%%%%%%%%%%%%%%%%%%%%%%%%%%%%%%%%%%%%%%%%%%%%
%%%%%%%%%%%%%%%%%%%%%%%%%%%%%%%%%%%%%%%%%%%%%%%%%%%%%%%%%%%%%%%%%%%%%%%%%%%%%%%%%

\section{Further convergence properties of Broyden's method}\label{sec_furtherconvprop}

In \Cref{thm_convsingularcase} we derived convergence properties
of Broyden's method that were helpful in establishing convergence of the Broyden matrices. In this section we draw further conclusions from \Cref{thm_convsingularcase}. 

\subsection{Violation of uniform linear independence}

In this section we prove that the normalized Broyden steps $(\hat s^k)$, defined at the beginning of \Cref{sec_main}, violate uniform linear independence. 
This is of interest because it shows that existing results on the convergence of the Broyden matrices cannot be applied in the setting of this work, see the related discussion in \Cref{sec_introconvBrmatrices} where uniform linear independence is defined.

\Cref{thm_convsingularcase}~1) shows that 
if $(u^k)$ is generated according to \Cref{thm_fundlemsingularJacobian} or \Cref{thm_starlikedomainofconvforBroyden}, then
$(\hat s^k)$ converges up to sign changes to $\phi$. Unsurprisingly, this prevents uniform linear independence. 

\begin{lemma}\label{lem_ULIviolated}
	Let $v\in\R^n$ for some $n\geq 2$ and let $(v^k)\subset\R^n$ be a sequence with $\norm{v^k}{2}=1$ for all $k\in\N_0$ and
	$\lim_{k\to\infty}\min\{\norm{v^k-v}{2},\norm{v^k+v}{2}\}=0$. Let for each $k\in\N_0$ an $n$-tuple $(k_1,\ldots,k_n)\in\N_0^n$ with $k\leq k_1<k_2<\ldots<k_n$ be given. Then the smallest singular value of 
	\begin{equation}
		\begin{pmatrix}
			v^{k_1} \, & \, \cdots \, & \, v^{k_n}
		\end{pmatrix}
	\end{equation}
	converges to zero for $k\to\infty$.\\
	In particular, if $(u^k)$ is generated according to \Cref{thm_fundlemsingularJacobian} or \Cref{thm_starlikedomainofconvforBroyden}, then $(\hat s^k)$ is not uniformly linearly independent.
\end{lemma}	

\begin{proof}
	Let us prove the main result first. For each $k\in\N_0$ define $\tilde v^k\in\{v^k,-v^k\}$ such that 
	$\min\{\norm{\tilde v^k-v}{2},\norm{\tilde v^k+v}{2}\}=\norm{\tilde v^k-v}{2}$.
	By multilinearity, there holds
	\begin{equation*}\label{eq_matrVk}
		\left\lvert 
		\det\Bigl(\begin{pmatrix}
			\tilde v^{k_1} \, & \, \tilde v^{k_2} \, & \, \cdots \, & \, \tilde v^{k_n}
		\end{pmatrix}\Bigr) \right\rvert = \bigl\lvert\det(V_k)\bigr\lvert,
	\end{equation*}
	where $V_k$ is the matrix given by \eqref{eq_matrVk}. 
	Since $\lim_{k\to\infty}\tilde v^k = v$, continuity entails $\lim_{k\to\infty} \lvert \det(V_k)\rvert =0$, where we used $n\geq 2$.
	This shows that at least one singular value of $V_k$ converges to zero. (In fact, it is not difficult to argue that $n-1$ singular values converge to zero.)
	
	The additional part concerning $(\hat s^k)$ follows if we apply the main result to $(v^k)\equiv (\hat s^k)$ and $v=\phi$, taking \Cref{thm_convsingularcase}~1) into account. 
\end{proof}	

\subsection{The nullspace of the Broyden limit matrix}

If $(u^k)$ is generated according to \Cref{thm_fundlemsingularJacobian} or \Cref{thm_starlikedomainofconvforBroyden}, then the convergence property 1) of \Cref{thm_convsingularcase} is satisfied and, by \Cref{thm_rlinconvofBrupdates}, the limit $\lim_{k\to\infty} B_k$ exists. 
This implies that $\lim_{k\to\infty} B_k$ is singular with $N$ in its nullspace as we show next.

\begin{lemma}\label{lem_singularityoflimitofBroydenmatrices}	
	Let $F:\R^n\rightarrow\R^n$ be strictly differentiable at $\uopt$. Let $(u^k)$ be generated by Algorithm~\ref{alg_broy} and let $\lim_{k\to\infty} u^k = \uopt$. Suppose there is $\phi\in\ker(F'(\uopt))$ with $\norm{\phi}{2}=1$ and  $\lim_{k\to\infty}\min\{\norm{\hat s^k - \phi}{2},\norm{\hat s^k + \phi}{2}\}=0$. Suppose further that $(B_k)$ converges to some $B$. Then $\spann{\phi}\subset\ker(B)$. 
	In particular, if $(u^k)$ is generated according to \Cref{thm_fundlemsingularJacobian} or \Cref{thm_starlikedomainofconvforBroyden}, then there holds 
	$N\subset \ker(\lim_{k\to\infty} B_k)$.
\end{lemma}

\begin{proof}
	We start the proof of the main claim by showing that $\lim_{k\to\infty}\norm{E_k\hat s^k}{2}= 0$.
	We have for all $k\in\N_0$ 
	\begin{equation*}
		\bigl[B_{k+1}-B_k\bigr] s^k
		= \Biggl[\left(y^k-B_k s^k\right)\frac{(s^k)^T}{\norm{s^k}{2}^2}\Biggr]s^k 
		= y^k-B_k s^k
		= R_{\uopt}^k \cdot \norm{s^k}{} - E_k s^k,
	\end{equation*}
	where
	\begin{equation*}
		R_{\uopt}^k:=\frac{F(u^{k+1})-F(u^k)-F'(\uopt)(u^{k+1}-u^k)}{\norm{u^{k+1}-u^k}{2}}.
	\end{equation*}
	Hence, 
	\begin{equation*}
		\norm{E_k\hat s^k}{2} 
		\leq \norm{B_{k+1}-B_k}{\ast} + \norm{R_{\uopt}^k}{2}.
	\end{equation*}
	From $\lim_{k\to\infty}\norm{B_{k+1}-B_k}{\ast} = 0 =\lim_{k\to\infty}\norm{R_{\uopt}^k}{2}$ we infer $\lim_{k\to\infty}\norm{E_k\hat s^k}{2}= 0$.
	The convergence of $(B_k)$ now implies 
	$\lim_{k\to\infty}\norm{(B-F'(\uopt)) \hat s^k}{2} = 0$, thus
	$(B-F'(\uopt))\phi = 0$, so $B\phi=0$, proving the main claim. 
	
	Regarding the additional claim (``In particular, ...'') we note that under the assumptions of either \Cref{thm_fundlemsingularJacobian} or  \Cref{thm_starlikedomainofconvforBroyden},
	$F$ is strictly differentiable at $\uopt$ due to \Cref{A2}.
	Moreover, there holds $\lim_{k\to\infty}\min\{\norm{\hat s^k - \phi}{2},\norm{\hat s^k + \phi}{2}\}=0$ by \Cref{thm_convsingularcase}~1), and 
	$(B_k)$ converges due to \Cref{thm_rlinconvofBrupdates}. 
	The assertion thus follows from the main claim.
\end{proof}

\begin{remark}
	In the numerical experiments $\ker(\lim_{k\to\infty} E_k)$ is consistently one-dimensional. This suggests that under the assumptions of \Cref{thm_convsingularcase}, $F'(\uopt)$ and $\lim_{k\to\infty} B_k$ usually coincide \emph{only} on $N$.
\end{remark}	

\subsection{A relationship between the Broyden updates and the Broyden steps}

\Cref{thm_convsingularcase}~3) shows that the norm of the Broyden updates satisfies $\varepsilon_k=O(\norm{s^k}{})$ for $k\to\infty$.
We now establish that under an additional assumption there also holds $\norm{s^k}{}=O(\varepsilon_k)$.
In particular, these relationships are satisfied if $(u^k)$ is generated 
according to \Cref{thm_fundlemsingularJacobian} or \Cref{thm_starlikedomainofconvforBroyden}.

\begin{lemma}\label{lem_convofdeltatotwo}
	Let $(u^k)$ be such that statements 1)--3) of \Cref{thm_convsingularcase} hold and suppose that there exists $\hat\kappa>0$ such that 
	\begin{equation*}
		\liminf_{k\to\infty} \, \frac{\norm{u^{k+1}-\uopt}{2}}{\norm{u^k-\uopt}{2}} \geq \hat\kappa.
	\end{equation*}
	Then $\lim_{k\to\infty}\delta_k=2$ for the sequence $(\delta_k)$ introduced in part 3) of \Cref{thm_convsingularcase}. 
	In particular, we have $\lim_{k\to\infty}\delta_k=2$ if $(u^k)$ is generated 
	according to \Cref{thm_fundlemsingularJacobian} or \Cref{thm_starlikedomainofconvforBroyden}.
\end{lemma}

\begin{proof}
	For the main claim we note that \Cref{thm_convsingularcase}~3) shows $\liminf_{k\to\infty}\delta_k\geq 2$. From the existence of $\hat\kappa$ we deduce that $\norm{u^{k+1}-\uopt}{}\geq \hat c\norm{s^k}{}$ for all sufficiently large $k$ and a constant $\hat c>0$. Therefore, \Cref{thm_convsingularcase}~2) 
	implies that there is $c>0$ such that for all large $k$ we have
	\begin{equation*}
		\norm{F(u^{k+1})}{2}\geq c\norm{u^{k+1}-\uopt}{2}^2
		\geq c\hat c^2\norm{s^k}{2}^2,
	\end{equation*}
	which yields $\limsup_{k\to\infty}\delta_k\leq 2$, thus  $\lim_{k\to\infty}\delta_k=2$.
	
	Regarding the additional claim (``In particular, ...'') we note that if $(u^k)$ is generated according to \Cref{thm_fundlemsingularJacobian} or \Cref{thm_starlikedomainofconvforBroyden}, then it satisfies 
	1)--3) of \Cref{thm_convsingularcase} and \eqref{eq_conclestforsingJac2}, in particular 
	\begin{equation*}
		\lim_{k\to\infty}\frac{\norm{u^{k+1}-\uopt}{2}}{\norm{u^k-\uopt}{2}} = \frac{\sqrt5 -1}{2}.
	\end{equation*}
	The assertion thus follows from the main claim.
\end{proof}

\subsection{A relationship between the r-factors of the errors and updates}

Using properties related to the statements 2) and 3) of \Cref{thm_convsingularcase}, we now establish a connection between
the r-factors of $(\norm{u^k-\uopt}{})$ and $(\varepsilon_k)$ that can be observed in the numerical experiments.

\begin{lemma}\label{lem_connectionrfactors}
	Let $F:\R^n\rightarrow\R^n$ and $\uopt\in\R^n$. Let $(u^k)$ be generated by Algorithm~\ref{alg_broy} and suppose that $\lim_{k\to\infty}u^k=\uopt$ and that there are constants $C,c,\hat c,\gamma,\hat\gamma,K>0$ such that
	for all $k\geq K$ there hold 
	\begin{equation*}
		\hat c\norm{u^{k+1}-\uopt}{2}^{1+\gamma}\leq\norm{F(u^{k+1})}{2}
		\qquad\text{ and }\qquad
		c\norm{s^k}{2}^{1+\hat\gamma}\leq\norm{F(u^{k+1})}{2}\leq C \norm{s^k}{2}^{1+\gamma}.
	\end{equation*}
	For $r^-:=\liminf_{k\to\infty}\sqrt[k]{\norm{u^{k}-\uopt}{2}}$ and   
	$r^+:=\limsup_{k\to\infty}\sqrt[k]{\norm{u^{k}-\uopt}{2}}$
	we have $\liminf_{k\to\infty}\sqrt[k]{\varepsilon_k}\geq(r^-)^{\hat\gamma}$ and
	$\limsup_{k\to\infty}\sqrt[k]{\varepsilon_k}\geq(r^+)^{\hat\gamma}$.
\end{lemma}

\begin{proof}
	There is $\bar C>0$ such that 
	$\norm{u^{k+1}-\uopt}{}\leq \bar C\norm{s^k}{}$ for all $k\geq K$.
	Combining this with $\varepsilon_k\geq c\norm{s^k}{}^{\hat \gamma}$ for all $k\geq K$, we obtain 
	$\liminf_{k\to\infty}\sqrt[k]{\varepsilon_k}\geq(r^-)^{\hat\gamma}$.
	The claim for $\limsup_{k\to\infty}\sqrt[k]{\varepsilon_k}$ follows similarly.
\end{proof}

%%%%%%%%%%%%%%%%%%%%%%%%%%%%%%%%%%%%%%%%%%%%%%%%%%%%%%%%%%%%%%%%%%%%%%%%%%%%%%%%%
%%%%%%%%%%%%%%%%%%%%%%%%%%%%%%%%%%%%%%%%%%%%%%%%%%%%%%%%%%%%%%%%%%%%%%%%%%%%%%%%%

\section{Numerical experiments}\label{sec_num}

We use numerical examples to verify our findings for Algorithm~\ref{alg_broypNl}. 
\Cref{sec_designofexperiments} presents the design of the experiments,  \Cref{sec_numexamples} the results. The \textsc{Julia} code used for \emph{cumulative runs} (explained in \Cref{sec_cumser} below) is available at \href{https://arxiv.org/abs/2111.12393}{https://arxiv.org/abs/2111.12393}.  

\subsection{Design of the experiments}\label{sec_designofexperiments}

\subsubsection{Implementation in Arbitrary Precision Arithmetic}

The experiments are carried out in \textsc{Julia~1.5.3} 
using Arbitrary Precision Arithmetic via the built-in type \texttt{BigFloat}. 
We replace the termination criterion $F(u^k)=0$ in Algorithm~\ref{alg_broypNl} by $\norm{F(u^k)}{2}\leq 10^{-\tol}$ with some prescribed exponent $\tol\in\{100,1000\}$ that we specify in \Cref{sec_cumser} below along with the precision. 
By $\bar k\in\N_0$ we denote the final value of $k$ in Algorithm~\ref{alg_broypNl}.

\subsubsection{Known solution and initialization}\label{sec_ID}

All examples have $\uopt=0$ as solution. 
The starting point $\hat u$ in Algorithm~\ref{alg_broypNl} is randomly chosen in $[-\alpha,\alpha]^n$, with $\alpha\in \{10^{-j}\}_{j=1}^5$. 
For $\hat B$, respectively, $B_0$ we use $\hat B=F'(\hat u)+\beta\norm{F'(\hat u)}{2}\hat R$ and $B_0=F'(u^0)+ \beta\norm{F'(u^0)}{2}R$, where $\hat R,R\in\R^{n\times n}$ are matrices with random entries in $[-1,1]$ and $\beta\in \{10^{-j}\}_{j=1}^5\cup\{0\}$. 
Note that for $\beta=0$ the first two steps are Newton steps.
We will also consider a variant of Algorithm~\ref{alg_broypNl} in the numerical experiments where $B_0$ is taken to be the Broyden update of $\hat B$ (based on $\hat u$ and $u^0$), even though \Cref{thm_starlikedomainofconvforBroyden} does not cover the use of Broyden's method with $(\hat u,\hat B)$ as initial guess. 

\subsubsection{Single runs and cumulative runs}\label{sec_cumser}

For each example we conduct one \emph{single run} and three series of \emph{cumulative runs}. 
For the single run we use $\tol=1000$ and a precision of 15000 decimal places.
Each series of cumulative runs consists of five cumulative runs, and each
cumulative run consists of $m$ single runs with varying initial data as described in \Cref{sec_ID}, where $m=10^6$ in the first example and $m=125000$ in all other examples.
The first cumulative series has $(\alpha,\beta)=10^{-j}(1,0)$, the second $(\alpha,\beta)=10^{-j}(1,1)$, and the third 
$(\alpha,\beta)=10^{-j}(1,0)$ but with $B_0$ obtained by a Broyden update; in all series we use $j=1,2,3,4,5$. 
The $m$ single runs that constitute a cumulative run are conducted with $\tol=100$ in the first three examples and with $\tol=1000$ in the fourth example; the precision is slightly higher than 1500 decimal digits in the first three examples and 2100 decimal digits in the last example.

\subsubsection{Quantities for single runs}\label{sec_qoI}

When displaying the results of a single run, we will denote 
the sequence $(\hat u,u^0,u^1,\ldots)$ generated by Algorithm~\ref{alg_broypNl}
as $(u^k)_{k\geq 0}$, i.e., the initial guess becomes $u^0$ and the iterates are indicated by $(u^k)_{k\geq 1}$. Similarly, $(\hat B, B^0,\ldots)$ is indicated by $(B_k)_{k\geq 0}$. 
With this notation in mind, the quantities that we display for single runs are as follows:
\begin{equation*}
	F_k:=F(u^k), 
	\qquad r_k := \sqrt[k]{\norm{u^k-\uopt}{2}},
	\qquad q_k := \frac{\norm{u^k-\uopt}{2}}{\norm{u^{k-1}-\uopt}{2}},
	\qquad \delta_k := \frac{\log(\norm{F_k}{2})}{\log(\norm{s^{k-1}}{2})},
\end{equation*}
as well as
\begin{equation*}
	R_k:=\sqrt[k]{\varepsilon_{k-1}}, \qquad Q_k := \frac{\varepsilon_{k-1}}{\varepsilon_{k-2}}, \qquad
	\zeta_k := \min\bigl\{\norm{\hat s^k-\phi}{2},\,\norm{\hat s^k+\phi}{2}\bigr\},
\end{equation*}
where $(\zeta_k)$ is only defined if \Cref{A2} holds. 
Undefined quantities are set to $-1$, e.g., $\delta_0=-1$.
The index shifts for $\varepsilon_k$ that appear in the definitions of $R_k$ and $Q_k$ are motivated by the fact that the related quantities $r_k$ and $q_k$ rely on $u^k$, not $u^{k+1}$, and that we want the same to be true for $R_k$ and $Q_k$, respectively; recall that $\varepsilon_{k}=\norm{F_{k+1}}{}/\norm{s^{k}}{}$ is based on $u^{k+1}$. We indicate by $0\leq\Lambda_1^k\leq\ldots\leq\Lambda_n^k$ the singular values of $E_k=B_k-F'(\uopt)$ and include $\Lambda_1^k$ and $\Lambda_2^k$ as well as $\norm{E_k}{2}=\Lambda_n^k$ in the results. 

\subsubsection{Quantities for cumulative runs}

Next we provide the quantities that we display to assess cumulative runs. We will take only those single runs within a cumulative run into account that converge to $\uopt=0$, measured by $\norm{u^{\bar k}}{2} \leq 10^{-10}$, and whose final q-factors $q_{\bar k}$ and $Q_{\bar k}$ belong to $[0.616,0.620]$ in Example~1 and 2, since we expect $q_{\bar k}$ and $Q_{\bar k}$ to be close to $0.618$ in the setting of \Cref{thm_starlikedomainofconvforBroyden}, see also \Cref{thm_qlinearrateforupdates}.
For reasons that will become clear, in Example~3 we demand $q_{\bar k}\in[0.753,0.757]$ and $Q_{\bar k}\in [0.568,0.572]$ instead. 
In example~4 no restriction is imposed on $q_{\bar k}$ and $Q_{\bar k}$. 
Numbering the $m$ single runs within a cumulative run by $1,\ldots,m$ we collect in $\calA\subset\{1,\ldots,m\}$ the indices of single runs that are taken into account.
To gauge if $\lim_{k\to\infty} q_k \approx 0.618$ we  
let $k_0(j):=\min\{\bar k(j)-25,0.75\bar k(j)\}$,  
$K(j):=\{k_0(j),k_0(j)+1,\ldots,\bar k(j)\}$, and display
\begin{equation}\label{eq_defqminusqplus}
	q^-:=\min_{j\in\calA}\min_{k\in K(j)}q_k^j
	\qquad\text{ and }\qquad
	q^+:=\max_{j\in\calA}\max_{k\in K(j)}q_k^j
\end{equation}
for each cumulative run. Here and in the following, we use an additional index $j$ to denote the single runs within a cumulative run. 
Further quantities that we provide are 
\begin{equation*}
	\norm{F}{2}^{-}:=\min_{j\in\calA} \, \norm{F(u^{\bar k(j)})}{2} \qquad\text{ and }\qquad
	\norm{F}{2}^{+}:=\max_{j\in\calA} \, \norm{F(u^{\bar k(j)})}{2}, 
\end{equation*}
\begin{equation*}
	\norm{u}{2}^{-}:=\min_{j\in\calA} \, \norm{u^{\bar k(j)}}{2}
	\qquad\text{ and }\qquad 
	\norm{u}{2}^+:=\max_{j\in\calA} \, \norm{u^{\bar k(j)}}{2}
\end{equation*}
as well as
\begin{equation*}
	\delta^- := \min_{j\in\calA}\min_{k\in K(j)} \delta_k^j\qquad\text{ and }\qquad
	\delta^+ := \max_{j\in\calA}\max_{k\in K(j)} \delta_k^j,
\end{equation*}
and for the singular values 
\begin{equation*}
	\Lambda_1 := \max_{j\in\calA}\Lambda_1^{\bar k(j)}\qquad\text{ and }\qquad
	\norm{E}{2}:=\min_{j\in\calA}\min_{k\in K(j)} \norm{E_{k}}{2}
\end{equation*}
as well as
\begin{equation*}
	\Lambda_2^- := \min_{j\in\calA}\Lambda_2^{\bar k(j)} \qquad\text{ and }\qquad
	\Lambda_2^+ := \max_{j\in\calA}\Lambda_2^{\bar k(j)}.
\end{equation*}
The r- and q-factor of convergence of $(u^k)$ and $(\varepsilon_k)$ are assessed via
\begin{equation*}
	r^- := \min_{j\in\calA} \, \min_{k\in K(j)}r_k^j \qquad\text{ and }\qquad
	r^+ := \max_{j\in\calA} \, \max_{k\in K(j)}r_k^j,
\end{equation*}
with analogue definitions for $q^-$ and $q^+$, see \eqref{eq_defqminusqplus}, as well as 
\begin{equation*}
	R^-:=\min_{j\in\calA}\, \min_{k\in K(j)}R_k^j \qquad\text{ and }\qquad
	R^+:=\max_{j\in\calA} \, \max_{k\in K(j)}R_k^j, 
\end{equation*}
with analogue definitions for $Q^-$ and $Q^+$. 
Convergence of $(\hat s^k)$ to $\phi$ is measured by
\begin{equation*}
	\zeta^- := \min_{j\in\calA} \, \max_{k\in K(j)}\zeta_k^j\qquad\text{ and }\qquad
	\zeta^+ := \max_{j\in\calA} \, \max_{k\in K(j)}\zeta_k^j.
\end{equation*}
By $\text{it}^{-}$ and $\text{it}^{+}$ we denote the minimal, respectively, maximal number of iterations that are required by the single runs within a cumulative run. If termination has not occurred after 500 iterations, the respective single run is not taken into account. 

\subsection{Numerical examples}\label{sec_numexamples}

\subsubsection{Example~1: Singular Jacobian with \texorpdfstring{\Cref{A2}}{Assumption~\ref{A2}} satisfied}

We consider the example from \cite{DeckerKelley_BroydenSingularJacobian}. Let
\begin{equation*}
	F:\R^2\rightarrow\R^2,\qquad
	F(u):=\begin{pmatrix}
		u_1 + u_2^2 \\
		\frac32 u_1 u_2 + u_2^2 + u_2^3
	\end{pmatrix}.
\end{equation*}
We have $F(0)=0$, and for $F'(0)$ we find $X=\spann{(1,0)^T}$ and $N=\spann{\phi}$, where $\phi:=(0,1)^T$. Since $P_N(F''(0)(\phi,\phi))=2\phi\neq 0$, \Cref{A2} holds.
The results of a single run with $(\alpha,\beta)=(0.01,0)$ are displayed in Table~\ref{tab_ex1_one}. They confirm, for instance, \Cref{thm_qlinearrateforupdates} in that $\lim_{k\to\infty}q_k=\lim_{k\to\infty} Q_k=\frac{\sqrt{5}-1}{2}$,
as well as \Cref{lem_convofdeltatotwo} in that $\lim_{k\to\infty}\delta_k=2$
and \Cref{thm_convsingularcase}~1) in that
$\lim_{k\to\infty}\zeta_k
=o(\norm{u^k-\uopt}{2})$.
Let us also mention that the convergence of 
all singular values $(\Lambda_j^k)$, $j\in\{1,\ldots,n\}$,
can be established by means of \cite[Theorem~3.2]{Griewank_87}, so the convergence of $(\norm{E_k}{})$ is expected. 
The same result implies that $(\Lambda_1^k)$ converges to zero. 
However, it is interesting to note that in all experiments, $\Lambda_1^k$ is the \emph{only} singular value that converges to zero.
Since the assumptions of \Cref{lem_singularityoflimitofBroydenmatrices} hold, this implies that $\lim_{k\to\infty} B_k$ agrees with $F'(\uopt)$ only on $N$. 

Table~\ref{tab_ex1_two} displays results for the three cumulative series described in \Cref{sec_cumser}. 
Though not depicted, all cumulative runs produced $\norm{F}{2}^-\approx 4\times 10^{-101}$, 
$\norm{u}{2}^-\approx 5\times 10^{-51}$ and $\norm{u}{2}^+\approx 8\times 10^{-51}$.
The results of Table~\ref{tab_ex1_two} are in line with the expected starlike domain of convergence with density~1 and support the theoretical findings of \Cref{sec_main}, for instance those pointed out for Table~\ref{tab_ex1_one}. In addition, we notice a strong connection
between the r-factors $r^-,r^+$ and their counterparts $R^-,R^+$ in that
$r^-\approx R^-$ with $R^->r^-$ and $r^+\approx R^+$ with $R^+>r^+$.
\Cref{lem_connectionrfactors} (with $\gamma=\hat\gamma=1$) provides assumptions under which these relations hold even when $(u^k)$ is not q-linearly convergent, which seems to occur for some larger values of $\alpha$, see Table~\ref{tab_ex1_two}.
Let us also note that if $\liminf_{k\to\infty}\delta_k>1$ and $(u^k)$ converges r-linearly, then $(\varepsilon_k)$ converges r-linearly, as follows from a slight generalization of \Cref{thm_rlinconvofBrupdates}; this situation seems to occur for larger values of $\alpha$. 
The third series of cumulative runs in Table~\ref{tab_ex1_two} shows that when Broyden's method Algorithm~\ref{alg_broy} is applied to $(\hat u,F'(\hat u))$, 
the results are similar as in the other two series.

\begin{table}
	\caption{Example~1: Single run with $(\alpha,\beta)=(0.01,0)$}
	\centering
	{\resizebox{0.9\textwidth}{!}{\input{tables/ex1s.tex}}}
	\label{tab_ex1_one}
\end{table}

\begin{table}
	\caption{Example~1: Cumulative runs with 	
		$(\alpha,\beta)=10^{-j}(1,0)$ (top five),
		$(\alpha,\beta)=10^{-j}(1,1)$ (middle five),
		and $(\alpha,\beta)=10^{-j}(1,0)$ with $B_0$ obtained by a Broyden update (bottom five)} 
	\centering
	{\resizebox{\textwidth}{!}{\input{tables/ex1c.tex}}}
	\label{tab_ex1_two}
\end{table}

Figure~\ref{fig_ex1_one} depicts the region of q-linear convergence of the iterates described in \Cref{thm_starlikedomainofconvforBroyden}. As expected, the region appears to be starlike and to have density one.

\begin{center}
	\begin{figure}
		\caption{Example~1: Domain $\calW_{\uopt}$ (blue) described in \Cref{thm_starlikedomainofconvforBroyden}, computed with $(\alpha,\beta)=10^{-j}(1,0)$, $j=1/2/3$ (left/middle/right); convergence that is likely faster or slower than described in 
			\Cref{thm_starlikedomainofconvforBroyden} is depicted in purple; starting points that do not converge within 300 iterations, including due to singularity, are shown in yellow; \Cref{thm_starlikedomainofconvforBroyden} asserts that the blue area includes a starlike domain with density one at $\uopt=0$}
		\smallskip
		\includegraphics[width=0.31\textwidth]{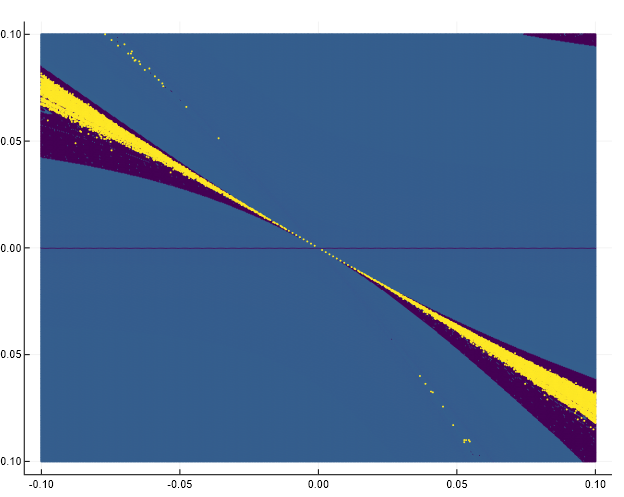}\hfill
		\includegraphics[width=0.31\textwidth]{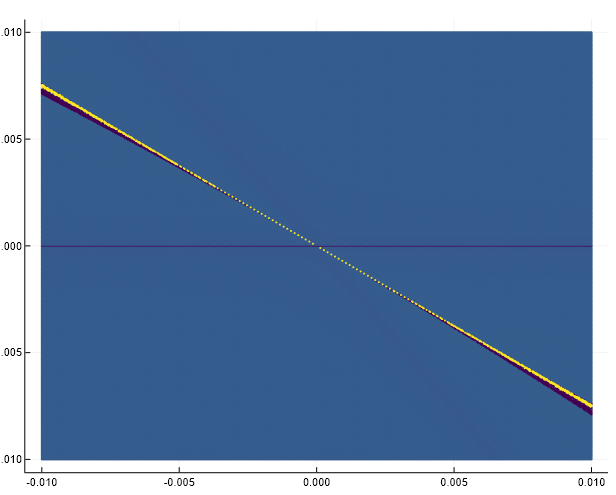}\hfill
		\includegraphics[width=0.31\textwidth]{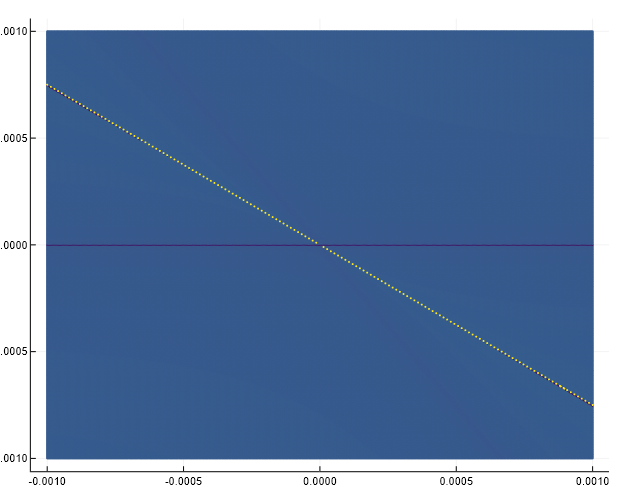}
		\label{fig_ex1_one}
	\end{figure}
\end{center}

\subsubsection{Example~2: Singular Jacobian with \texorpdfstring{\Cref{A2}}{Assumption~\ref{A2}} satisfied}

Let 
\begin{equation*}
	F:\R^3\rightarrow\R^3,\qquad
	F(u):=\begin{pmatrix}
		u_1^2 + u_2 + u_3 \\
		u_2-2 u_3^3 \\
		5 u_3 + u_3^2
	\end{pmatrix}.
\end{equation*}
Since we have $F(0)=0$, $X=\spann{\{(1,1,0)^T,(1,0,5)^T\}}$, $\ker(F'(0))=\spann{\phi}$, where $\phi:=(1,0,0)^T$, and $P_N(F''(0)(\phi,\phi))=2\phi\neq 0$, \Cref{A2} is satisfied at $\uopt=0$. 

The results of a single run with $(\alpha,\beta)=(0.1,0.1)$ are displayed in Table~\ref{tab_ex2_one}. We note that neither $(q_k)$ nor $(Q_k)$ are convergent, but that $(r_k)$ and $(R_k)$ are. 
The values of $(r_k)$ together with the apparent $\lim_{k\to\infty} \delta_k=2$ imply r-linear convergence of $(\varepsilon_k)$ by \Cref{thm_rlinconvofBrupdates}; Table~\ref{tab_ex2_one} confirms this. In contrast, \Cref{thm_qlinearrateforupdates} does not apply and it is apparent that $(\varepsilon_k)$ is indeed not q-linearly convergent since $Q_k>1$ for large values of $k$.
Since we have observed this behavior many times in the numerical experiments and since an inspection of $(q_k)$ (including the values not depicted) reveals that $q_k\in[0.53,0.72]$ for all large $k$, we suspect that it may be possible to prove asymptotic r-linear or non-asymptotic q-linear convergence of $(u^k)$ 
under assumptions that are weaker than those of \Cref{thm_fundlemsingularJacobian}.
The relationship $\lim_{k\to\infty}\zeta_k=o(\norm{u^k-\uopt}{2})$ from 
\Cref{thm_convsingularcase}~1) does not hold; 
since $q_k\leq 0.72$ for all large $k$, we infer that the second relation in \eqref{eq_conclestforsingJac} must be violated.

Table~\ref{tab_ex2_two} shows the three series of cumulative runs described in \Cref{sec_cumser}. 
In addition, we consistently find $\norm{F}{2}^-\approx 4\times 10^{-101}$, $\norm{u}{2}^-\approx 6.2\times 10^{-51}$ and $\norm{u}{2}^+\approx 1\times 10^{-50}$.
The results of Table~\ref{tab_ex2_two} are in agreement with the theoretical results of this work. 
For instance, in all three series the values of $\zeta^+$ imply that when $\alpha$ is sufficiently small, we have $\lim_{k\to\infty}\zeta_k=o(\norm{u^k-\uopt}{2})$. 

\begin{table}
	\caption{Example~2: Single run with $(\alpha,\beta)=(0.1,0.1)$}
	\centering
	{\resizebox{0.9\textwidth}{!}{\input{tables/ex2s.tex}}}
	\label{tab_ex2_one}
\end{table}

\begin{table}
	\caption{Example~2: The same three series of cumulative runs as in Example~1}
	\centering
	{\resizebox{\textwidth}{!}{\input{tables/ex2c.tex}}}
	\label{tab_ex2_two}
\end{table}

\subsubsection{Example~3: Singular Jacobian with \texorpdfstring{\Cref{A2}}{Assumption~\ref{A2}} violated}

We modify $F$ from Example~2 so that it violates \Cref{A2}. 
Let
\begin{equation*}
	F:\R^3\rightarrow\R^3,\qquad
	F(u):=\begin{pmatrix}
		u_1^3 + u_2 + u_3 \\
		u_2-2 u_3^3 \\
		5 u_3 + u_3^2
	\end{pmatrix}.
\end{equation*}
This modification does not affect $F'(0)$, but yields $P_N(F''(0)(\phi,\phi))=0$, so the last requirement of \Cref{A2} is violated. In the sense of \cite[Section~2]{Griewank_StarlikedomainsofconvNewtSingSiamRev},
$\uopt=0$ is a \emph{regular second order singularity}. Theoretical results for Broyden's method on this type of singularity are not available.

The results of a single run of Algorithm~\ref{alg_broypNl} with $(\alpha,\beta)=(0.1,0)$, but with $B_0$ is obtained by a Broyden update, are shown in Table~\ref{tab_ex3_one}. 
Table~\ref{tab_ex3_two} displays the three cumulative series described in \Cref{sec_cumser}, where we always find $\norm{F}{2}^-\approx 4\times 10^{-101}$, $\norm{u}{2}^-\approx 3.5\times 10^{-34}$ and $\norm{u}{2}^+\approx 4.6\times 10^{-34}$.
As differences to Example~2 we note that $\lim_{k\to\infty}\delta_k=3$, $\lim_{k\to\infty}q_k\approx 0.75488$ is the real root of $t^3+t^2-1$, and 
$\lim_{k\to\infty} Q_k\approx 0.5698$ is the square of this root.
These limits are to be expected; the first by Taylor expansion of $F$, the second 
by the convergence rate of Broyden's method for $n=1$ (note that 
the asymptotic rate $\lim_{k\to\infty}q_k=\frac{\sqrt{5}-1}{2}$ that holds under \Cref{A2} is also the same as for $n=1$), see \cite{Diez},
and the third after straightforward modifications in the proof of \Cref{lem_qlinearrateofupdates}, see also \cite[Lemma~5~2.]{Man_convBrmatressoneD}. 
From a high-level view, however, the behavior is quite similar to the previous examples and we expect that the results of this work can be generalized to the case that $P_N(F^{(M)}(\uopt)(\phi,\ldots,\phi))=0$ for all $1\leq M\leq m$ and some $m\geq 1$ with $P_N(F^{(m+1)}(\uopt)(\phi,\ldots,\phi))\neq 0$.

\begin{table}
	\caption{Example~3: Single run with $(\alpha,\beta)=(0.1,0)$
		and $B_0$ obtained by a Broyden update}
	\centering
	{\resizebox{0.9\textwidth}{!}{\input{tables/ex3s.tex}}}
	\label{tab_ex3_one}
\end{table}

\begin{table}
	\caption{Example~3: The same three series of cumulative runs as in Example 1 and 2}
	\centering
	{\resizebox{\textwidth}{!}{\input{tables/ex3c.tex}}}
	\label{tab_ex3_two}
\end{table}

\subsubsection{Example~4: Regular Jacobian}

For comparison we consider an example in which $F'(\uopt)$ is regular. Let 
\begin{equation*}
	F:\R^3\rightarrow\R^3,\qquad
	F(u):=\begin{pmatrix}
		(1+u_1)^2(1+u_2)+(1+u_2)^2+u_3-2 \\ e^{u_1} + (1+u_2)^3 + u_3^2 - 2 \\ e^{u_3^2} + (1+u_2)^2 - 2
	\end{pmatrix}.
\end{equation*}
The results of a single run are displayed in Table~\ref{tab_ex4_one}, while the usual series of cumulative runs is shown in Table~\ref{tab_ex4_two}. 
The results differ significantly from those obtained in the previous examples. 
For instance, the values of $q^-$ and $q^+$ depict q-superlinear convergence of $(u^k)$ to $\uopt$, which is further underlined by the iteration numbers $\text{it}^-$ and $\text{it}^+$ that are much lower in this example. 
The values of $Q^+$ indicate that $(\varepsilon_k)$ is not q-linearly convergent, in general. In Table~\ref{tab_ex4_one} we notice that $(\delta_k)$ does not converge. 
It is unknown whether the iterates of Broyden's method admit a q-order larger than one, but in this example this seems to be the case:
The bound $\delta^-\approx 1.13$ suggests a q-order of approximately $1.13$ for $(\norm{u^k-\uopt}{2})$, which implies an r-order of $1.13$ for $(\varepsilon_k)$ by \cite[Lemma~3]{Man_convBrmatressoneD}. 
Perhaps not coincidentally, 
it is known that $(\norm{u^k-\uopt}{2})$ has an r-order of at least 
$2^{1/(2n)}\approx 1.122$ (this follows from Gay's theorem \cite[Theorem~3.1]{ConvOfBroydensMethodOnLinearSystems_Gay}).
No single runs had to be removed from the cumulative runs except for $(\alpha,\beta)=(0.1,0.1)$, where 1883 runs were removed. 

\begin{table}
	\caption{Example~4: Single run with $(\alpha,\beta)=(0.1,0)$}
	\centering
	{\resizebox{0.9\textwidth}{!}{\input{tables/ex4s.tex}}}
	\label{tab_ex4_one}
\end{table}

\begin{table}
	\caption{Example~4: The same three series of cumulative runs as in Example 1--3}
	\centering
	{\resizebox{\textwidth}{!}{\input{tables/ex4c.tex}}}
	\label{tab_ex4_two}
\end{table}

%%%%%%%%%%%%%%%%%%%%%%%%%%%%%%%%%%%%%%%%%%%%%%%%%%%%%%%%%%%%%%%%%%%%%%%%%%%%%%%
%%%%%%%%%%%%%%%%%%%%%%%%%%%%%%%%%%%%%%%%%%%%%%%%%%%%%%%%%%%%%%%%%%%%%%%%%%%%%%%

\section{Summary}\label{sec_conclusion}

We have shown that for certain singular equations the combination of Broyden's method with a preceding Newton--like step guarantees convergence for all starting points in a starlike domain with density~1; likewise for some accelerated schemes.
We have proven that for these problems the Broyden matrices converge and the Broyden steps violate uniform linear independence. 
Numerical experiments confirmed the results and suggested some extensions.
The code is freely available. 

%%%%%%%%%%%%%%%%%%%%%%%%%%%%%%%%%%%%%%%%%%%%%%%%%%%%%%%%%%%%%%%%%%%%%%%%%%%%%%%%%
%%%%%%%%%%%%%%%%%%%%%%%%%%%%%%%%%%%%%%%%%%%%%%%%%%%%%%%%%%%%%%%%%%%%%%%%%%%%%%%%%

\bibliographystyle{alphaurl}
\bibliography{lit}

\end{document}

%% file: tables/ex1s.tex
\begin{tabular}{llllllllllll}
$k$ & $\lvert|F_k\rvert|$ & $\lvert|u^k\rvert|$ & $r_k$ & $q_k$ & $\varepsilon_{k-1}$ & $R_k$ & $Q_k$ & $\delta_k$ & $\zeta_k$ & $\Lambda_1^k$ & $\lvert|E_k\rvert|$ \\
\hline 
0 & 4.8e-03 & 5.2e-03 & 0.005 & -1 & -1 & -1 & -1 & -1 & 1.4 & 9.2e-04 & 1.2e-02 \\
1 & 2.5e-06 & 1.6e-03 & 0.040 & 0.307 & 7.9e-03 & 0.089 & -1 & 2.415 & 3.0e-03 & 1.6e-03 & 4.8e-03 \\
2 & 8.9e-07 & 7.9e-04 & 0.093 & 0.500 & 1.1e-03 & 0.104 & 2.138 & 1.951 & 5.7e-06 & 1.5e-03 & 3.8e-03 \\
3 & 3.9e-07 & 5.3e-04 & 0.152 & 0.666 & 1.5e-03 & 0.197 & 1.332 & 1.790 & 2.5e-06 & 1.1e-03 & 2.8e-03 \\
4 & 1.4e-07 & 3.2e-04 & 0.200 & 0.600 & 6.7e-04 & 0.232 & 0.451 & 1.863 & 8.6e-07 & 7.9e-04 & 2.5e-03 \\
5 & 5.5e-08 & 2.0e-04 & 0.241 & 0.625 & 4.7e-04 & 0.278 & 0.694 & 1.849 & 3.3e-07 & 5.0e-04 & 2.4e-03 \\
6 & 2.1e-08 & 1.2e-04 & 0.276 & 0.615 & 2.8e-04 & 0.310 & 0.591 & 1.864 & 1.2e-07 & 3.2e-04 & 2.4e-03 \\
7 & 8.0e-09 & 7.5e-05 & 0.305 & 0.619 & 1.7e-04 & 0.339 & 0.629 & 1.868 & 4.8e-08 & 2.0e-04 & 2.4e-03 \\
8 & 3.1e-09 & 4.7e-05 & 0.330 & 0.618 & 1.1e-04 & 0.362 & 0.614 & 1.875 & 1.8e-08 & 1.2e-04 & 2.4e-03 \\
9 & 1.2e-09 & 2.9e-05 & 0.351 & 0.618 & 6.6e-05 & 0.382 & 0.620 & 1.880 & 6.9e-09 & 7.5e-05 & 2.4e-03 \\
10 & 4.5e-10 & 1.8e-05 & 0.370 & 0.618 & 4.1e-05 & 0.399 & 0.617 & 1.885 & 2.6e-09 & 4.7e-05 & 2.4e-03 \\
\vdots & \vdots & \vdots & \vdots & \vdots & \vdots & \vdots & \vdots & \vdots & \vdots & \vdots & \vdots \\
500 & 7.0e-215 & 7.0e-108 & 0.611 & 0.618 & 1.6e-107 & 0.612 & 0.618 & 1.995 & 4.1e-214 & 1.8e-107 & 2.4e-03 \\
501 & 2.7e-215 & 4.3e-108 & 0.611 & 0.618 & 9.9e-108 & 0.612 & 0.618 & 1.995 & 1.6e-214 & 1.1e-107 & 2.4e-03 \\
502 & 1.0e-215 & 2.7e-108 & 0.611 & 0.618 & 6.1e-108 & 0.612 & 0.618 & 1.995 & 6.0e-215 & 7.0e-108 & 2.4e-03 \\
503 & 3.9e-216 & 1.7e-108 & 0.611 & 0.618 & 3.8e-108 & 0.612 & 0.618 & 1.995 & 2.3e-215 & 4.3e-108 & 2.4e-03 \\
\vdots & \vdots & \vdots & \vdots & \vdots & \vdots & \vdots & \vdots & \vdots & \vdots & \vdots & \vdots \\
1000 & 7.2e-424 & 2.2e-212 & 0.615 & 0.618 & 5.1e-212 & 0.615 & 0.618 & 1.997 & 4.2e-423 & 5.9e-212 & 2.4e-03 \\
1001 & 2.7e-424 & 1.4e-212 & 0.615 & 0.618 & 3.2e-212 & 0.615 & 0.618 & 1.997 & 1.6e-423 & 3.6e-212 & 2.4e-03 \\
1002 & 1.0e-424 & 8.6e-213 & 0.615 & 0.618 & 2.0e-212 & 0.615 & 0.618 & 1.997 & 6.2e-424 & 2.3e-212 & 2.4e-03 \\
1003 & 4.0e-425 & 5.3e-213 & 0.615 & 0.618 & 1.2e-212 & 0.615 & 0.618 & 1.997 & 2.4e-424 & 1.4e-212 & 2.4e-03 \\
\vdots & \vdots & \vdots & \vdots & \vdots & \vdots & \vdots & \vdots & \vdots & \vdots & \vdots & \vdots \\
1500 & 7.4e-633 & 7.2e-317 & 0.616 & 0.618 & 1.7e-316 & 0.616 & 0.618 & 1.998 & 4.4e-632 & 1.9e-316 & 2.4e-03 \\
1501 & 2.8e-633 & 4.5e-317 & 0.616 & 0.618 & 1.0e-316 & 0.616 & 0.618 & 1.998 & 1.7e-632 & 1.2e-316 & 2.4e-03 \\
1502 & 1.1e-633 & 2.8e-317 & 0.616 & 0.618 & 6.3e-317 & 0.616 & 0.618 & 1.998 & 6.4e-633 & 7.2e-317 & 2.4e-03 \\
1503 & 4.1e-634 & 1.7e-317 & 0.616 & 0.618 & 3.9e-317 & 0.616 & 0.618 & 1.998 & 2.4e-633 & 4.5e-317 & 2.4e-03 \\
\vdots & \vdots & \vdots & \vdots & \vdots & \vdots & \vdots & \vdots & \vdots & \vdots & \vdots & \vdots \\
2000 & 7.6e-842 & 2.3e-421 & 0.616 & 0.618 & 5.3e-421 & 0.617 & 0.618 & 1.999 & 4.5e-841 & 6.1e-421 & 2.4e-03 \\
2001 & 2.9e-842 & 1.4e-421 & 0.616 & 0.618 & 3.3e-421 & 0.617 & 0.618 & 1.999 & 1.7e-841 & 3.7e-421 & 2.4e-03 \\
2002 & 1.1e-842 & 8.8e-422 & 0.616 & 0.618 & 2.0e-421 & 0.617 & 0.618 & 1.999 & 6.5e-842 & 2.3e-421 & 2.4e-03 \\
2003 & 4.2e-843 & 5.5e-422 & 0.616 & 0.618 & 1.3e-421 & 0.617 & 0.618 & 1.999 & 2.5e-842 & 1.4e-421 & 2.4e-03 \\
\vdots & \vdots & \vdots & \vdots & \vdots & \vdots & \vdots & \vdots & \vdots & \vdots & \vdots & \vdots \\
2378 & 7.7e-1000 & 2.3e-500 & 0.617 & 0.618 & 5.3e-500 & 0.617 & 0.618 & 1.999 & 4.5e-999 & 6.1e-500 & 2.4e-03 \\
2379 & 2.9e-1000 & 1.4e-500 & 0.617 & 0.618 & 3.3e-500 & 0.617 & 0.618 & 1.999 & 1.7e-999 & 3.8e-500 & 2.4e-03 \\
2380 & 1.1e-1000 & 8.9e-501 & 0.617 & 0.618 & 2.0e-500 & 0.617 & 0.618 & 1.999 & 6.6e-1000 & 2.3e-500 & 2.4e-03 \\
2381 & 4.3e-1001 & 5.5e-501 & 0.617 & 0.618 & 1.3e-500 & 0.617 & 0.618 & 1.999 & 2.5e-1000 & 1.4e-500 & 2.4e-03 \\
\end{tabular}

%% file: tables/ex1c.tex
\begin{tabular}{lllllllllllllllllll}
$j$ %& $\lvert|F\rvert|^-$ 
& $\lvert|F\rvert|^+$ %& $\lvert|u|\rvert^-$ & $\lvert|u|\rvert^+$ 
& $r^-$ & $r^+$ & $q^-$ & $q^+$ & $R^-$ & $R^+$ & $Q^-$ & $Q^+$ & $\delta^-$ & $\delta^+$ & $\zeta^-$ & $\zeta^+$ & $\Lambda_1$ %& $\Lambda_2^-$ & $\Lambda_2^+$ 
& $\lvert|E|\rvert$ %& $\varnothing it$ 
& $\text{it}^-$ & $\text{it}^+$ 
%& $\text{conv. to } 0$ & conv. & rem. s. 
& rem. \\\toprule
1 %& 3.1e-101 
& 1e-100 %& 4.5e-51 & 8.5e-51 
& 0.531 & 0.796 & 6e-04 & 3e+02 & 0.549 & 0.798 & 1e-05 & 9e+04 & 1.857 & 2.108 & 7e-104 & 2e-43 & 2e-54 %& 1e-07 & 4e+00 
& 1e-07 %& 232.9 
& 207 & 479 %& 988302 & 990373 & 0 
& 11698 \\
2 %& 3.8e-101 
& 1e-100 %& 5.2e-51 & 8.4e-51 
& 0.562 & 0.740 & 0.02 & 14 & 0.565 & 0.742 & 0.002 & 196 & 1.891 & 2.046 & 4e-103 & 3e-44 & 7e-53 %& 1.1e-08 & 1.6e+00 
& 1e-08 %& 228.4 
& 204 & 367 %& 998865 & 999072 & 0 
& 1135 \\
3 %& 3.8e-101 
& 1e-100 %& 5.2e-51 & 8.4e-51 
& 0.547 & 0.739 & 0.056 & 19 & 0.550 & 0.741 & 0.003 & 359 & 1.946 & 2.051 & 2e-101 & 3e-45 & 6e-51 %& 3.7e-10 & 7.2e-01 
& 4e-10 %& 223.7 
& 196 & 368 %& 999889 & 999916 & 0 
& 111 \\
4 %& 3.8e-101 
& 1e-100 %& 5.2e-51 & 8.4e-51 
& 0.533 & 0.639 & 0.6179 & 0.6182 & 0.536 & 0.641 & 0.6175 & 0.6186 & 1.987 & 1.989 & 2e-100 & 3e-48 & 3e-51 %& 1.9e-11 & 4.4e-01 
& 2e-11 %& 218.9 
& 189 & 255 %& 999990 & 999992 & 0 
& 10 \\
5 %& 3.8e-101 
& 1e-100 %& 5.2e-51 & 8.4e-51 
& 0.529 & 0.615 & 0.6180 & 0.6180 & 0.532 & 0.617 & 0.6180 & 0.6180 & 1.987 & 1.989 & 2e-100 & 3e-51 & 1e-50 
%& 6e-12 & 2.8e-01 
& 6e-12 %& 214.1 
& 187 & 237 %& 1000000 & 1000000 & 0 
& 0 \\\midrule
1 %& 3.8e-101 
& 1e-100 %& 5.2e-51 & 8.4e-51 
& 0.534 & 0.803 & 4e-04 & 550 & 0.533 & 0.804 & 2e-05 & 6e+04 & 1.866 & 2.103 & 2e-103 & 1e-42 & 5e-56 %& 3.9e-05 & 5.6e+00 
& 4e-05 %& 234.2 
& 194 & 496 
%& 976194 & 979397 & 0 
& 23806 \\
2 %& 3.8e-101 
& 1e-100 %& 5.2e-51 & 8.4e-51 
& 0.565 & 0.775 & 0.002 & 77 & 0.568 & 0.776 & 2e-04 & 5e+03 & 1.897 & 2.078 & 4e-103 & 8e-44 & 5e-57 %& 4.4e-06 & 1.6e+00 
& 4e-06 %& 229.5 
& 205 & 435 %& 997955 & 998341 & 0 
& 2045 \\
3 %& 3.8e-101 
& 1e-100 %& 5.2e-51 & 8.4e-51 
& 0.555 & 0.656 & 0.24 & 1 & 0.558 & 0.658 & 0.004 & 63 & 1.905 & 2.002 & 7e-102 & 2e-44 & 1e-56 %& 1.0e-07 & 1.6e+00 
& 1e-07 %& 224.7 
& 200 & 270 %& 999810 & 999837 & 0 
& 190 \\
4 %& 3.8e-101 
& 1e-100 %& 5.2e-51 & 8.4e-51 
& 0.546 & 0.637 & 0.6166 & 0.6195 & 0.549 & 0.640 & 0.6166 & 0.6194 & 1.987 & 1.989 & 2e-101 & 4e-47 & 2e-56 %& 7.5e-08 & 3.1e-01 
& 8e-08 %& 219.9 
& 195 & 254 %& 999983 & 999985 & 0 
& 17 \\
5 %& 3.8e-101 
& 1e-100 %& 5.2e-51 & 8.4e-51 
& 0.530 & 0.615 & 0.6180 & 0.6180 & 0.533 & 0.617 & 0.6180 & 0.6180 & 1.987 & 1.989 & 2e-100 & 7e-53 & 9e-57 %& 8.3e-09 & 2.6e-01 
& 8e-09 %& 215.1 
& 187 & 237 %& 999996 & 999996 & 0 
& 4 \\\midrule
1 %& 3.8e-101 
& 1e-100 %& 5.2e-51 & 8.4e-51 
& 0.569 & 0.801 & 0.001 & 651 & 0.566 & 0.803 & 2e-06 & 5e+05 & 1.863 & 2.129 & 1e-103 & 2e-42 & 5e-55 %& 5.5e-07 & 2.6e+00 
& 6e-07 %& 234.1 
& 208 & 490 %& 980355 & 980672 & 0 
& 19645 \\
2 %& 3.8e-101 
& 1e-100 %& 5.2e-51 & 8.4e-51 
& 0.558 & 0.737 & 0.2812 & 5 & 0.560 & 0.739 & 0.014 & 26 & 1.926 & 2.025 & 6e-101 & 1e-45 & 2e-54 %& 5.6e-09 & 5.7e-01 
& 6e-09 %& 229.4 
& 201 & 364 %& 998852 & 998855 & 0 
& 1148 \\
3 %& 3.8e-101 
& 1e-100 %& 5.2e-51 & 8.4e-51 
& 0.542 & 0.635 & 0.6162 & 0.6199 & 0.545 & 0.638 & 0.6142 & 0.6222 & 1.987 & 1.989 & 2e-100 & 8e-47 & 5e-54 %& 1.0e-10 & 3.5e-01 
& 1e-10 %& 224.6 
& 193 & 252 %& 999938 & 999938 & 0 
& 62 \\
4 %& 3.8e-101 
& 1e-100 %& 5.2e-51 & 8.4e-51 
& 0.543 & 0.626 & 0.6180 & 0.6180 & 0.546 & 0.629 & 0.6180 & 0.6180 & 1.987 & 1.989 & 2e-100 & 3e-69 & 1e-54 %& 1.3e-10 & 8.1e-02 
& 1e-10 %& 219.8 
& 194 & 245 %& 999999 & 999999 & 0 
& 1 \\
5 %& 3.8e-101 
& 1e-100 %& 5.2e-51 & 8.4e-51 
& 0.534 & 0.622 & 0.6180 & 0.6180 & 0.536 & 0.624 & 0.6180 & 0.6180 & 1.987 & 1.989 & 2e-100 & 3e-84 & 9e-54 %& 1.2e-11 & 6.2e-02 
& 1e-11 %& 215.0 
& 189 & 242 %& 1000000 & 1000000 & 0 
& 0 \\
\end{tabular}

%% file: tables/ex2s.tex
\begin{tabular}{lllllllllllll}
$k$ & $\lvert|F_k\rvert|$ & $\lvert|u^k\rvert|$ & $r_k$ & $q_k$ & $\varepsilon_{k-1}$ & $R_k$ & $Q_k$ & $\delta_k$ & $\zeta_k$ & $\Lambda_1^k$ & $\Lambda_2^k$ & $\lvert|E_k\rvert|$ \\\toprule
0 & 2.2e-01 & 1.3e-01 & 0.133 & -1 & -1 & -1 & -1 & -1 & 7.1e-01 & 2.4e-03 & 2.2e-01 & 8.9e-01 \\
1 & 9.2e-02 & 2.7e-02 & 0.163 & 0.199 & 8.1e-01 & 0.902 & -1 & 1.088 & 4.7e-01 & 8.6e-02 & 1.5e-01 & 8.2e-01 \\
2 & 1.9e-02 & 2.7e-02 & 0.300 & 1.020 & 5.3e-01 & 0.809 & 0.643 & 1.192 & 5.4e-01 & 1.3e-02 & 1.3e-01 & 6.4e-01 \\
3 & 2.2e-03 & 4.8e-02 & 0.467 & 1.763 & 8.1e-02 & 0.534 & 0.153 & 1.699 & 1.3e-01 & 5.3e-02 & 1.1e-01 & 6.4e-01 \\
4 & 4.0e-02 & 1.0e-01 & 0.637 & 2.199 & 2.6e-01 & 0.766 & 3.245 & 1.709 & 1.0e-01 & 4.8e-02 & 1.1e-01 & 6.0e-01 \\
5 & 5.1e-03 & 7.9e-02 & 0.655 & 0.754 & 2.8e-02 & 0.550 & 0.105 & 3.114 & 2.0e-02 & 2.1e-02 & 1.1e-01 & 6.0e-01 \\
6 & 5.9e-02 & 2.5e-01 & 0.821 & 3.189 & 3.4e-01 & 0.856 & 12.2 & 1.619 & 2.3e-02 & 1.1e-01 & 2.8e-01 & 6.3e-01 \\
7 & 3.4e-03 & 6.1e-02 & 0.704 & 0.240 & 1.8e-02 & 0.605 & 0.053 & 3.438 & 2.0e-02 & 1.1e-01 & 2.6e-01 & 6.2e-01 \\
8 & 2.4e-03 & 4.9e-02 & 0.715 & 0.806 & 2.0e-01 & 0.838 & 11.4 & 1.357 & 4.2e-03 & 9.2e-02 & 1.1e-01 & 6.0e-01 \\
9 & 8.4e-04 & 2.7e-02 & 0.696 & 0.545 & 3.8e-02 & 0.721 & 0.186 & 1.859 & 7.5e-03 & 6.5e-02 & 1.1e-01 & 6.0e-01 \\
\vdots & \vdots & \vdots & \vdots & \vdots & \vdots & \vdots & \vdots & \vdots & \vdots & \vdots & \vdots\\
500 & 2.0e-209 & 2.2e-105 & 0.618 & 0.560 & 1.1e-104 & 0.620 & 0.561 & 1.992 & 1.3e-104 & 5.4e-105 & 1.1e-01 & 6.0e-01 \\
501 & 1.8e-210 & 1.5e-105 & 0.618 & 0.652 & 2.3e-105 & 0.619 & 0.199 & 1.996 & 2.4e-105 & 3.2e-105 & 1.1e-01 & 6.0e-01 \\
502 & 1.8e-210 & 1.0e-105 & 0.618 & 0.715 & 4.4e-105 & 0.620 & 1.909 & 1.990 & 2.4e-105 & 2.2e-105 & 1.1e-01 & 6.0e-01 \\
503 & 6.9e-211 & 7.0e-106 & 0.619 & 0.671 & 2.0e-105 & 0.620 & 0.459 & 1.993 & 1.9e-105 & 1.5e-105 & 1.1e-01 & 6.0e-01 \\
\vdots & \vdots & \vdots & \vdots & \vdots & \vdots & \vdots & \vdots & \vdots & \vdots & \vdots & \vdots\\
1000 & 7.7e-421 & 4.8e-211 & 0.616 & 0.576 & 2.2e-210 & 0.617 & 0.468 & 1.996 & 2.6e-210 & 1.1e-210 & 1.1e-01 & 6.0e-01 \\
1001 & 8.5e-422 & 3.2e-211 & 0.616 & 0.676 & 5.5e-211 & 0.617 & 0.252 & 1.997 & 1.3e-211 & 6.9e-211 & 1.1e-01 & 6.0e-01 \\
1002 & 9.7e-422 & 2.3e-211 & 0.617 & 0.718 & 1.1e-210 & 0.618 & 1.938 & 1.995 & 6.3e-211 & 4.8e-211 & 1.1e-01 & 6.0e-01 \\
1003 & 2.6e-422 & 1.5e-211 & 0.617 & 0.649 & 3.2e-211 & 0.617 & 0.296 & 1.997 & 3.8e-211 & 3.3e-211 & 1.1e-01 & 6.0e-01 \\
\vdots & \vdots & \vdots & \vdots & \vdots & \vdots & \vdots & \vdots & \vdots & \vdots & \vdots & \vdots\\
1500 & 2.7e-632 & 1.0e-316 & 0.616 & 0.593 & 3.9e-316 & 0.616 & 0.397 & 1.998 & 4.8e-316 & 2.3e-316 & 1.1e-01 & 6.0e-01 \\
1501 & 5.2e-633 & 6.9e-317 & 0.616 & 0.692 & 1.7e-316 & 0.616 & 0.428 & 1.998 & 3.4e-317 & 1.5e-316 & 1.1e-01 & 6.0e-01 \\
1502 & 4.5e-633 & 5.0e-317 & 0.616 & 0.713 & 2.2e-316 & 0.617 & 1.337 & 1.997 & 1.4e-316 & 1.0e-316 & 1.1e-01 & 6.0e-01 \\
1503 & 8.6e-634 & 3.1e-317 & 0.616 & 0.631 & 4.7e-317 & 0.616 & 0.211 & 1.999 & 6.8e-317 & 7.0e-317 & 1.1e-01 & 6.0e-01 \\
\vdots & \vdots & \vdots & \vdots & \vdots & \vdots & \vdots & \vdots & \vdots & \vdots & \vdots & \vdots\\
2000 & 9.0e-844 & 2.1e-422 & 0.616 & 0.612 & 6.6e-422 & 0.616 & 0.326 & 1.998 & 8.4e-422 & 4.8e-422 & 1.1e-01 & 6.0e-01 \\
2001 & 3.0e-844 & 1.5e-422 & 0.616 & 0.705 & 4.8e-422 & 0.616 & 0.725 & 1.998 & 1.9e-422 & 3.1e-422 & 1.1e-01 & 6.0e-01 \\
2002 & 2.0e-844 & 1.1e-422 & 0.616 & 0.704 & 4.5e-422 & 0.616 & 0.926 & 1.998 & 3.1e-422 & 2.2e-422 & 1.1e-01 & 6.0e-01 \\
2003 & 3.2e-845 & 6.5e-423 & 0.616 & 0.613 & 7.7e-423 & 0.616 & 0.174 & 1.999 & 1.2e-422 & 1.5e-422 & 1.1e-01 & 6.0e-01 \\
\vdots & \vdots & \vdots & \vdots & \vdots & \vdots & \vdots & \vdots & \vdots & \vdots & \vdots & \vdots\\
2369 & 8.4e-1000 & 2.8e-500 & 0.615 & 0.590 & 4.3e-500 & 0.616 & 0.271 & 1.999 & 3.0e-500 & 6.5e-500 & 1.1e-01 & 6.0e-01 \\
2370 & 9.3e-1000 & 1.5e-500 & 0.615 & 0.537 & 7.2e-500 & 0.616 & 1.655 & 1.999 & 4.1e-500 & 3.7e-500 & 1.1e-01 & 6.0e-01 \\
2371 & 3.2e-1000 & 8.1e-501 & 0.615 & 0.545 & 4.6e-500 & 0.616 & 0.646 & 1.998 & 4.8e-500 & 2.0e-500 & 1.1e-01 & 6.0e-01 \\
2372 & 3.5e-1001 & 5.1e-501 & 0.615 & 0.632 & 1.2e-500 & 0.616 & 0.253 & 1.999 & 1.5e-500 & 1.1e-500 & 1.1e-01 & 6.0e-01 \\
\end{tabular}

%% file: tables/ex2c.tex
\begin{tabular}{lllllllllllllllll}
$j$ %& $\lvert|F\rvert|^-$ 
& $\lvert|F\rvert|^+$ %& $\lvert|u|\rvert^-$ & $\lvert|u|\rvert^+$ & $r^-$ & $r^+$ 
& $q^-$ & $q^+$ %& $R^-$ & $R^+$ 
& $Q^-$ & $Q^+$ & $\delta^-$ & $\delta^+$ & $\zeta^-$ & $\zeta^+$ & $\Lambda_1$ & $\Lambda_2^-$ & $\Lambda_2^+$ & $\lvert|E|\rvert$ %& $\varnothing it$ 
& $\text{it}^-$ & $\text{it}^+$ %& $\text{conv. to } 0$ & conv. & rem. s. 
& rem. \\\toprule 
1 %& 3.8e-101 
& 1e-100 %& 6.2e-51 & 1.0e-50 & 0.570 & 0.639 
& 0.6180 & 0.6180 %& 0.572 & 0.640 
& 0.6180 & 0.6180 & 1.991 & 1.992 & 2e-1579 & 1e-286 & 3e-86 & 4e-19 & 3e-03 & 1e-13 %& 232.7 
& 208 & 255 %& 125000 & 125000 & 0 
& 0 \\
2 %& 3.8e-101 
& 1e-100 %& 6.2e-51 & 1.0e-50 & 0.563 & 0.627 
& 0.6180 & 0.6180 %& 0.565 & 0.629 
& 0.6180 & 0.6180 & 1.991 & 1.992 & 1e-1809 & 9e-474 & 2e-91 & 1e-22 & 4e-06 & 4e-16 %& 227.5 
& 204 & 246 %& 125000 & 125000 & 0 
& 0 \\
3 %& 3.8e-101 
& 1e-100 %& 6.2e-51 & 1.0e-50 & 0.558 & 0.626 
& 0.6180 & 0.6180 %& 0.559 & 0.627 
& 0.6180 & 0.6180 & 1.991 & 1.992 & 5e-1864 & 1e-634 & 6e-92 & 1e-23 & 4e-09 & 1e-16 %& 222.7 
& 201 & 245 %& 125000 & 125000 & 0 
& 0 \\
4 %& 3.8e-101 
& 1e-100 %& 6.2e-51 & 1.0e-50 & 0.545 & 0.616 
& 0.6180 & 0.6180 %& 0.547 & 0.617 
& 0.6180 & 0.6180 & 1.991 & 1.992 & 3e-2380 & 1e-761 & 6e-105 & 2e-31 & 4e-12 & 7e-22 %& 217.9 
& 194 & 237 %& 125000 & 125000 & 0 
& 0 \\
5 %& 3.8e-101 
& 1e-100 %& 6.2e-51 & 1.0e-50 & 0.533 & 0.607 
& 0.6180 & 0.6180 %& 0.535 & 0.608 
& 0.6180 & 0.6180 & 1.991 & 1.992 & 5e-2236 & 1e-924 & 9e-101 & 6e-30 & 4e-15 & 6e-21 %& 213.1 
& 188 & 230 %& 125000 & 125000 & 0 
& 0 \\\midrule
1 %& 3.8e-101 
& 1e-100 %& 6.1e-51 & 1.0e-50 & 0.517 & 0.750 
& 6e-03 & 89 %& 0.532 & 0.751 
& 1e-04 & 6842 & 1.868 & 2.084 & 2e-324 & 2e-43 & 3e-57 & 1e-05 & 3 & 2e-02 %& 236.9 
& 196 & 385 %& 124585 & 124602 & 0 
& 415 \\
2 %& 3.8e-101 
& 1e-100 %& 6.2e-51 & 1.0e-50 & 0.577 & 0.636 
& 0.6180 & 0.6180 %& 0.579 & 0.637 
& 0.6180 & 0.6180 & 1.991 & 1.992 & 5e-413 & 4e-50 & 4e-59 & 1e-06 & 8e-01 & 5e-03 %& 231.4 
& 212 & 252 %& 124994 & 124995 & 0 
& 6 \\
3 %& 3.8e-101 
& 1e-100 %& 6.2e-51 & 1.0e-50 & 0.566 & 0.631 
& 0.6180 & 0.6180 %& 0.568 & 0.632 
& 0.6180 & 0.6180 & 1.991 & 1.992 & 1e-514 & 6e-113 & 2e-56 & 4e-07 & 8e-02 & 7e-04 %& 226.4 
& 206 & 248 %& 125000 & 125000 & 0 
& 0 \\
4 %& 3.8e-101 
& 1e-100 %& 6.2e-51 & 1.0e-50 & 0.555 & 0.626 
& 0.6180 & 0.6180 %& 0.556 & 0.627 
& 0.6180 & 0.6180 & 1.991 & 1.992 & 3e-594 & 2e-266 & 2e-55 & 2e-07 & 2e-03 & 5e-05 %& 221.5 
& 199 & 244 %& 125000 & 125000 & 0 
& 0 \\
5 %& 3.8e-101 
& 1e-100 %& 6.2e-51 & 1.0e-50 & 0.547 & 0.619 
& 0.6180 & 0.6180 %& 0.549 & 0.621 
& 0.6180 & 0.6180 & 1.991 & 1.992 & 5e-671 & 4e-351 & 1e-55 & 4e-08 & 5e-05 & 5e-06 %& 216.7 
& 195 & 240 %& 125000 & 125000 & 0 
& 0 \\\midrule
1 %& 3.8e-101 
& 1e-100 %& 6.2e-51 & 1.0e-50 & 0.579 & 0.650 
& 0.6180 & 0.6180 %& 0.581 & 0.651 
& 0.6180 & 0.6180 & 1.991 & 1.992 & 5e-842 & 7e-104 & 2e-66 & 5e-08 & 2e-01 & 3e-05 %& 233.6 
& 213 & 263 %& 125000 & 125000 & 0 
& 0 \\
2 %& 3.8e-101 
& 1e-100 %& 6.2e-51 & 1.0e-50 & 0.575 & 0.632 
& 0.6180 & 0.6180 %& 0.576 & 0.633 
& 0.6180 & 0.6180 & 1.991 & 1.992 & 7e-809 & 2e-218 & 1e-65 & 1e-07 & 1e-02 & 9e-06 %& 228.4 
& 211 & 249 %& 125000 & 125000 & 0 
& 0 \\
3 %& 3.8e-101 
& 1e-100 %& 6.2e-51 & 1.0e-50 & 0.554 & 0.627 
& 0.6180 & 0.6180 %& 0.555 & 0.628 
& 0.6180 & 0.6180 & 1.991 & 1.992 & 5e-1023 & 8e-304 & 5e-70 & 2e-09 & 9e-04 & 1e-07 %& 223.6 
& 199 & 245 %& 125000 & 125000 & 0 
& 0 \\
4 %& 3.8e-101 
& 1e-100 %& 6.2e-51 & 1.0e-50 & 0.554 & 0.619 
& 0.6180 & 0.6180 %& 0.555 & 0.621 
& 0.6180 & 0.6180 & 1.991 & 1.992 & 9e-1123 & 2e-385 & 2e-71 & 2e-10 & 9e-05 & 1e-07 %& 218.8 
& 199 & 240 %& 125000 & 125000 & 0 
& 0 \\
5 %& 3.8e-101 
& 1e-100 %& 6.2e-51 & 1.0e-50 & 0.524 & 0.612 
& 0.6180 & 0.6180 %& 0.526 & 0.613 
& 0.6180 & 0.6180 & 1.991 & 1.992 & 3e-1255 & 3e-445 & 9e-71 & 2e-12 & 9e-06 & 2e-09 %& 213.9 
& 184 & 234 %& 125000 & 125000 & 0 
& 0 \\
\end{tabular}

%% file: tables/ex3s.tex
\begin{tabular}{lllllllllllll}
$k$ & $\lvert|F_k\rvert|$ & $\lvert|u^k\rvert|$ & $r_k$ & $q_k$ & $\varepsilon_{k-1}$ & $R_k$ & $Q_k$ & $\delta_k$ & $\zeta_k$ & $\Lambda_1^k$ & $\Lambda_2^k$ & $\lvert|E_k\rvert|$ \\\toprule 
0 & 3.5e-01 & 1.1e-01 & 0.109 & -1 & -1 & -1 & -1 & -1 & 7.1e-01 & 0 & 2.1e-02 & 1.4e-01 \\
1 & 5.4e-03 & 1.6e-01 & 0.406 & 1.509 & 5.0e-02 & 0.224 & -1 & 2.349 & 7.6e-03 & 3.3e-04 & 3.8e-02 & 1.2e-01 \\
2 & 3.1e-03 & 7.8e-02 & 0.427 & 0.473 & 3.5e-02 & 0.328 & 0.706 & 2.367 & 9.2e-02 & 6.1e-04 & 4.6e-02 & 1.1e-01 \\
3 & 3.3e-04 & 6.7e-02 & 0.510 & 0.865 & 3.1e-02 & 0.421 & 0.886 & 1.761 & 2.3e-03 & 6.1e-04 & 1.6e-02 & 1.1e-01 \\
4 & 1.3e-04 & 5.0e-02 & 0.549 & 0.736 & 7.4e-03 & 0.375 & 0.237 & 2.217 & 2.3e-03 & 6.0e-04 & 1.1e-02 & 1.1e-01 \\
5 & 5.5e-05 & 3.8e-02 & 0.579 & 0.760 & 4.7e-03 & 0.409 & 0.628 & 2.212 & 1.7e-04 & 5.8e-04 & 6.1e-03 & 1.1e-01 \\
6 & 2.3e-05 & 2.8e-02 & 0.602 & 0.755 & 2.4e-03 & 0.423 & 0.524 & 2.285 & 7.2e-05 & 5.1e-04 & 3.9e-03 & 1.1e-01 \\
7 & 1.0e-05 & 2.1e-02 & 0.619 & 0.754 & 1.4e-03 & 0.441 & 0.585 & 2.320 & 8.2e-06 & 4.0e-04 & 2.9e-03 & 1.1e-01 \\
8 & 4.3e-06 & 1.6e-02 & 0.633 & 0.755 & 8.1e-04 & 0.453 & 0.568 & 2.356 & 2.4e-06 & 2.7e-04 & 2.4e-03 & 1.1e-01 \\
%9 & 1.8e-06 & 1.2e-02 & 0.644 & 0.755 & 4.6e-04 & 0.464 & 0.570 & 2.390 & 3.4e-07 & 1.6e-04 & 2.3e-03 & 1.1e-01 \\
%10 & 7.9e-07 & 9.2e-03 & 0.653 & 0.755 & 2.6e-04 & 0.473 & 0.570 & 2.419 & 8.1e-08 & 9.6e-05 & 2.2e-03 & 1.1e-01 \\
%11 & 3.4e-07 & 7.0e-03 & 0.661 & 0.755 & 1.5e-04 & 0.480 & 0.570 & 2.446 & 1.3e-08 & 5.5e-05 & 2.2e-03 & 1.1e-01 \\
%12 & 1.5e-07 & 5.3e-03 & 0.668 & 0.755 & 8.6e-05 & 0.487 & 0.570 & 2.470 & 2.9e-09 & 3.1e-05 & 2.2e-03 & 1.1e-01 \\
%13 & 6.3e-08 & 4.0e-03 & 0.674 & 0.755 & 4.9e-05 & 0.492 & 0.570 & 2.493 & 5.0e-10 & 1.8e-05 & 2.2e-03 & 1.1e-01 \\
%14 & 2.7e-08 & 3.0e-03 & 0.679 & 0.755 & 2.8e-05 & 0.497 & 0.570 & 2.513 & 1.0e-10 & 1.0e-05 & 2.2e-03 & 1.1e-01 \\
%15 & 1.2e-08 & 2.3e-03 & 0.683 & 0.755 & 1.6e-05 & 0.501 & 0.570 & 2.532 & 1.9e-11 & 5.8e-06 & 2.2e-03 & 1.1e-01 \\
\vdots & \vdots & \vdots & \vdots & \vdots & \vdots & \vdots & \vdots & \vdots & \vdots & \vdots & \vdots & \vdots \\
500 & 2.4e-186 & 1.3e-62 & 0.752 & 0.755 & 5.5e-124 & 0.568 & 0.570 & 2.977 & 5.8e-360 & 2.0e-124 & 2.2e-03 & 1.1e-01 \\
501 & 1.0e-186 & 1.0e-62 & 0.752 & 0.755 & 3.1e-124 & 0.568 & 0.570 & 2.977 & 1.1e-360 & 1.1e-124 & 2.2e-03 & 1.1e-01 \\
502 & 4.4e-187 & 7.6e-63 & 0.752 & 0.755 & 1.8e-124 & 0.568 & 0.570 & 2.977 & 2.1e-361 & 6.5e-125 & 2.2e-03 & 1.1e-01 \\
%503 & 1.9e-187 & 5.7e-63 & 0.752 & 0.755 & 1.0e-124 & 0.568 & 0.570 & 2.977 & 4.1e-362 & 3.7e-125 & 2.2e-03 & 1.1e-01 \\
\vdots & \vdots & \vdots & \vdots & \vdots & \vdots & \vdots & \vdots & \vdots & \vdots & \vdots & \vdots & \vdots \\
1500 & 1.0e-552 & 1.0e-184 & 0.754 & 0.755 & 3.1e-368 & 0.569 & 0.570 & 2.992 & 1.5e-1078 & 1.1e-368 & 2.2e-03 & 1.1e-01 \\
1501 & 4.4e-553 & 7.6e-185 & 0.754 & 0.755 & 1.8e-368 & 0.569 & 0.570 & 2.992 & 2.9e-1079 & 6.5e-369 & 2.2e-03 & 1.1e-01 \\
1502 & 1.9e-553 & 5.7e-185 & 0.754 & 0.755 & 1.0e-368 & 0.569 & 0.570 & 2.992 & 5.6e-1080 & 3.7e-369 & 2.2e-03 & 1.1e-01 \\
%1503 & 8.1e-554 & 4.3e-185 & 0.754 & 0.755 & 5.8e-369 & 0.569 & 0.570 & 2.992 & 1.1e-1080 & 2.1e-369 & 2.2e-03 & 1.1e-01 \\
\vdots & \vdots & \vdots & \vdots & \vdots & \vdots & \vdots & \vdots & \vdots & \vdots & \vdots & \vdots & \vdots \\
2721 & 4.7e-1000 & 7.8e-334 & 0.754 & 0.755 & 1.9e-666 & 0.569 & 0.570 & 2.996 & 6.4e-1956 & 6.8e-667 & 2.2e-03 & 1.1e-01 \\
2722 & 2.0e-1000 & 5.9e-334 & 0.754 & 0.755 & 1.1e-666 & 0.569 & 0.570 & 2.996 & 1.2e-1956 & 3.9e-667 & 2.2e-03 & 1.1e-01 \\
2723 & 8.6e-1001 & 4.4e-334 & 0.754 & 0.755 & 6.0e-667 & 0.569 & 0.570 & 2.996 & 2.3e-1957 & 2.2e-667 & 2.2e-03 & 1.1e-01 \\
\end{tabular}

%% file: tables/ex3c.tex
\begin{tabular}{lllllllllllllllll}
$j$ %& $\lvert|F\rvert|^-$ 
& $\lvert|F\rvert|^+$ %& $\lvert|u|\rvert^-$ & $\lvert|u|\rvert^+$ & $r^-$ & $r^+$ 
& $q^-$ & $q^+$ %& $R^-$ & $R^+$ 
& $Q^-$ & $Q^+$ & $\delta^-$ & $\delta^+$ & $\zeta^-$ & $\zeta^+$ & $\Lambda_1$ & $\Lambda_2^-$ & $\Lambda_2^+$ & $\lvert|E|\rvert$ %& $\varnothing it$ 
& $\text{it}^-$ & $\text{it}^+$ %& $\text{conv. to } 0$ & conv. & rem. s. 
& rem. \\\toprule 
1 %& 4.3e-101 
& 1e-100 %& 3.5e-34 & 4.6e-34 & 0.710 & 0.809 
& 0.7549 & 0.7549 %& 0.507 & 0.657 
& 0.5698 & 0.5698 & 2.952 & 2.957 & 4e-1766 & 1e-321 & 3e-104 & 2e-19 & 4e-03 & 2e-13 %& 267.0 
& 228 & 355 %& 125000 & 125000 & 0 
& 0 \\
2 %& 4.3e-101 
& 1e-100 %& 3.5e-34 & 4.6e-34 & 0.704 & 0.798 
& 0.7549 & 0.7549 %& 0.498 & 0.639 
& 0.5698 & 0.5698 & 2.952 & 2.957 & 1e-1858 & 4e-594 & 1e-104 & 1e-22 & 8e-06 & 1e-14 %& 265.2 
& 223 & 333 %& 125000 & 125000 & 0 
& 0 \\
3 %& 4.3e-101 
& 1e-100 %& 3.5e-34 & 4.6e-34 & 0.694 & 0.803 
& 0.7549 & 0.7549 %& 0.485 & 0.647 
& 0.5698 & 0.5698 & 2.952 & 2.957 & 2e-2080 & 3e-836 & 1e-111 & 6e-27 & 8e-09 & 6e-17 %& 264.7 
& 215 & 342 %& 124999 & 124999 & 0 
& 1 \\
4 %& 4.3e-101 
& 1e-100 %& 3.5e-34 & 4.6e-34 & 0.695 & 0.802 
& 0.7549 & 0.7549 %& 0.485 & 0.646 
& 0.5698 & 0.5698 & 2.952 & 2.957 & 2e-2506 & 2e-1076 & 4e-118 & 5e-34 & 8e-12 & 4e-21 %& 264.5 
& 216 & 341 %& 125000 & 125000 & 0 
& 0 \\
5 %& 4.3e-101 
& 1e-100 %& 3.5e-34 & 4.6e-34 & 0.683 & 0.807 
& 0.7549 & 0.7549 %& 0.469 & 0.654 
& 0.5698 & 0.5698 & 2.952 & 2.957 & 1e-2375 & 3e-1316 & 1e-126 & 2e-34 & 8e-15 & 1e-20 %& 264.5 
& 207 & 350 %& 125000 & 125000 & 0 
& 0 \\\midrule
1 %& 4.3e-101 
& 1e-100 %& 3.5e-34 & 4.9e-34 & 0.715 & 0.861 
& 0.06 & 12 %& 0.514 & 0.743 
& 1e-04 & 1643 & 2.682 & 3.107 & 7e-527 & 2e-58 & 4e-77 & 2e-06 & 880 & 2e-02 %& 268.1 
& 232 & 491 %& 123784 & 123846 & 0 
& 1216 \\
2 %& 4.3e-101 
& 1e-100 %& 3.5e-34 & 4.6e-34 & 0.700 & 0.842 
& 0.75 & 0.76 %& 0.492 & 0.711 
& 0.56 & 0.58 & 2.952 & 2.957 & 2e-508 & 2e-63 & 5e-76 & 8e-07 & 48 & 8e-04 %& 260.0 
& 220 & 430 %& 124951 & 124958 & 0 
& 49 \\
3 %& 4.3e-101 
& 1e-100 %& 3.5e-34 & 4.6e-34 & 0.694 & 0.766 
& 0.7543 & 0.7554 %& 0.484 & 0.589 
& 0.5676 & 0.5721 & 2.952 & 2.957 & 1e-646 & 9e-65 & 3e-75 & 5e-08 & 4e-01 & 2e-04 %& 251.8 
& 215 & 286 %& 124999 & 125000 & 0 
& 1 \\
4 %& 4.3e-101 
& 1e-100 %& 3.5e-34 & 4.6e-34 & 0.680 & 0.757 
& 0.7549 & 0.7549 %& 0.465 & 0.575 
& 0.5698 & 0.5698 & 2.952 & 2.957 & 2e-709 & 7e-297 & 9e-72 & 3e-10 & 5e-04 & 2e-05 %& 243.6 
& 205 & 275 %& 124999 & 124999 & 0 
& 1 \\
5 %& 4.3e-101 
& 1e-100 %& 3.5e-34 & 4.6e-34 & 0.661 & 0.766 
& 0.7549 & 0.7549 %& 0.440 & 0.589 
& 0.5698 & 0.5698 & 2.952 & 2.957 & 5e-822 & 9e-235 & 1e-73 & 2e-11 & 7e-03 & 2e-06 %& 235.4 
& 193 & 285 %& 125000 & 125000 & 0 
& 0 \\\midrule
1 %& 4.3e-101 
& 1e-100 %& 3.5e-34 & 4.6e-34 & 0.717 & 0.754 
& 0.7549 & 0.7549 %& 0.517 & 0.571 
& 0.5698 & 0.5698 & 2.952 & 2.957 & 2e-1248 & 2e-130 & 5e-99 & 3e-09 & 2e-01 & 2e-05 %& 263.1 
& 234 & 271 %& 125000 & 125000 & 0 
& 0 \\
2 %& 4.3e-101 
& 1e-100 %& 3.5e-34 & 4.6e-34 & 0.703 & 0.748 
& 0.7549 & 0.7549 %& 0.497 & 0.561 
& 0.5698 & 0.5698 & 2.952 & 2.957 & 7e-1566 & 3e-248 & 5e-103 & 7e-13 & 4e-03 & 7e-08 %& 256.1 
& 222 & 263 %& 125000 & 125000 & 0 
& 0 \\
3 %& 4.3e-101 
& 1e-100 %& 3.5e-34 & 4.6e-34 & 0.700 & 0.742 
& 0.7549 & 0.7549 %& 0.492 & 0.553 
& 0.5698 & 0.5698 & 2.952 & 2.957 & 7e-1695 & 5e-365 & 1e-103 & 4e-13 & 7e-05 & 7e-08 %& 249.0 
& 220 & 257 %& 124999 & 124999 & 0 
& 1 \\
4 %& 4.3e-101 
& 1e-100 %& 3.5e-34 & 4.6e-34 & 0.663 & 0.734 
& 0.7549 & 0.7549 %& 0.442 & 0.541 
& 0.5698 & 0.5698 & 2.952 & 2.957 & 9e-1903 & 1e-468 & 2e-109 & 3e-15 & 2e-06 & 1e-09 %& 241.3 
& 194 & 247 %& 124999 & 124999 & 0 
& 1 \\
5 %& 4.3e-101 
& 1e-100 %& 3.5e-34 & 4.6e-34 & 0.663 & 0.726 
& 0.7549 & 0.7549 %& 0.443 & 0.529 
& 0.5698 & 0.5698 & 2.952 & 2.957 & 1e-2109 & 3e-553 & 9e-115 & 3e-16 & 6e-08 & 6e-11 %& 233.4 
& 194 & 239 %& 125000 & 125000 & 0 
& 0 \\
\end{tabular}

%% file: tables/ex4s.tex
\begin{tabular}{llllllllllll}
$k$ & $\lvert|F_k\rvert|$ & $\lvert|u^k\rvert|$ & $r_k$ & $q_k$ & $\varepsilon_{k-1}$ & $R_k$ & $Q_k$ & $\delta_k$ & $\Lambda_1^k$ & $\Lambda_2^k$ & $\lvert|E_k\rvert|$ \\
\hline 
0 & 3.0e-01 & 7.5e-02 & 7.5e-02 & -1 & -1 & -1 & -1 & -1 & 1.8e-02 & 1.0e-01 & 3.4e-01 \\
1 & 1.2e-02 & 3.8e-03 & 6.1e-02 & 5.1e-02 & 3.7e-01 & 6.1e-01 & -1 & 1.734 & 2.0e-07 & 1.6e-03 & 1.4e-02 \\
2 & 2.3e-05 & 8.7e-06 & 2.1e-02 & 2.3e-03 & 6.0e-03 & 1.8e-01 & 3.9e-02 & 1.916 & 2.0e-07 & 1.2e-03 & 1.0e-02 \\
3 & 3.5e-08 & 1.5e-08 & 1.1e-02 & 1.8e-03 & 4.0e-03 & 2.5e-01 & 6.7e-01 & 1.474 & 3.6e-07 & 5.5e-04 & 9.3e-03 \\
4 & 1.1e-11 & 3.0e-12 & 5.0e-03 & 2.0e-04 & 7.2e-04 & 2.4e-01 & 1.8e-01 & 1.402 & 1.2e-09 & 5.5e-04 & 9.3e-03 \\
5 & 1.0e-14 & 2.8e-15 & 3.7e-03 & 9.2e-04 & 3.4e-03 & 3.9e-01 & 4.7 & 1.214 & 5.6e-14 & 5.5e-04 & 8.6e-03 \\
6 & 1.2e-19 & 3.7e-20 & 1.7e-03 & 1.3e-05 & 4.2e-05 & 2.4e-01 & 1.2e-02 & 1.300 & 5.4e-17 & 5.5e-04 & 8.6e-03 \\
7 & 4.4e-23 & 1.2e-23 & 1.4e-03 & 3.3e-04 & 1.2e-03 & 4.3e-01 & 28 & 1.150 & 1.9e-22 & 5.5e-04 & 8.6e-03 \\
8 & 1.4e-26 & 3.9e-27 & 1.2e-03 & 3.2e-04 & 1.2e-03 & 4.7e-01 & 9.8e-01 & 1.128 & 2.3e-25 & 5.5e-04 & 8.5e-03 \\
9 & 1.4e-34 & 3.8e-35 & 3.6e-04 & 9.8e-09 & 3.6e-08 & 1.8e-01 & 3.1e-05 & 1.282 & 7.3e-29 & 5.5e-04 & 8.5e-03 \\
10 & 9.7e-45 & 2.6e-45 & 8.9e-05 & 7.0e-11 & 2.6e-10 & 1.3e-01 & 7.2e-03 & 1.279 & 7.1e-37 & 5.5e-04 & 8.5e-03 \\
11 & 3.9e-55 & 1.1e-55 & 2.6e-05 & 4.0e-11 & 1.5e-10 & 1.5e-01 & 5.7e-01 & 1.221 & 5.0e-47 & 5.5e-04 & 8.5e-03 \\
12 & 1.9e-67 & 5.3e-68 & 6.7e-06 & 5.0e-13 & 1.8e-12 & 1.3e-01 & 1.2e-02 & 1.213 & 2.0e-57 & 5.5e-04 & 8.5e-03 \\
13 & 6.8e-79 & 1.9e-79 & 2.4e-06 & 3.5e-12 & 1.3e-11 & 1.7e-01 & 7.1 & 1.162 & 9.9e-70 & 5.5e-04 & 8.5e-03 \\
14 & 3.3e-90 & 8.9e-91 & 9.9e-07 & 4.8e-12 & 1.8e-11 & 1.9e-01 & 1.4 & 1.137 & 3.5e-81 & 5.5e-04 & 8.5e-03 \\
15 & 3.8e-110 & 1.0e-110 & 1.3e-07 & 1.2e-20 & 4.3e-20 & 6.2e-02 & 2.4e-09 & 1.215 & 1.7e-92 & 5.5e-04 & 8.5e-03 \\
16 & 1.8e-138 & 5.0e-139 & 7.3e-09 & 4.8e-29 & 1.8e-28 & 2.3e-02 & 4.1e-09 & 1.252 & 2.0e-112 & 5.5e-04 & 8.5e-03 \\
17 & 2.1e-169 & 5.6e-170 & 4.0e-10 & 1.1e-31 & 4.1e-31 & 2.1e-02 & 2.3e-03 & 1.220 & 9.4e-141 & 5.5e-04 & 8.5e-03 \\
18 & 7.7e-205 & 2.1e-205 & 1.7e-11 & 3.8e-36 & 1.4e-35 & 1.5e-02 & 3.4e-05 & 1.206 & 1.1e-171 & 5.5e-04 & 8.5e-03 \\
19 & 1.6e-238 & 4.4e-239 & 1.2e-12 & 2.1e-34 & 7.7e-34 & 2.2e-02 & 56 & 1.162 & 4.0e-207 & 5.5e-04 & 8.5e-03 \\
20 & 3.5e-272 & 9.5e-273 & 1.1e-13 & 2.2e-34 & 7.9e-34 & 2.7e-02 & 1.0 & 1.139 & 8.3e-241 & 5.5e-04 & 8.5e-03 \\
21 & 6.8e-325 & 1.8e-325 & 1.7e-15 & 1.9e-53 & 7.1e-53 & 4.3e-03 & 9.0e-20 & 1.192 & 1.8e-274 & 5.5e-04 & 8.5e-03 \\
22 & 1.8e-401 & 4.9e-402 & 3.6e-18 & 2.7e-77 & 9.8e-77 & 5.0e-04 & 1.4e-24 & 1.234 & 3.5e-327 & 5.5e-04 & 8.5e-03 \\
23 & 1.1e-491 & 3.0e-492 & 3.3e-21 & 6.0e-91 & 2.2e-90 & 1.8e-04 & 2.2e-14 & 1.223 & 9.3e-404 & 5.5e-04 & 8.5e-03 \\
24 & 1.9e-593 & 5.2e-594 & 1.9e-24 & 1.8e-102 & 6.4e-102 & 9.0e-05 & 2.9e-12 & 1.206 & 5.6e-494 & 5.5e-04 & 8.5e-03 \\
25 & 1.7e-694 & 4.7e-695 & 2.0e-27 & 9.0e-102 & 3.3e-101 & 1.4e-04 & 5.1 & 1.169 & 9.8e-596 & 5.5e-04 & 8.5e-03 \\
26 & 1.5e-795 & 4.1e-796 & 3.5e-30 & 8.8e-102 & 3.2e-101 & 1.9e-04 & 9.8e-01 & 1.145 & 8.8e-697 & 5.5e-04 & 8.5e-03 \\
27 & 2.2e-936 & 6.1e-937 & 3.7e-34 & 1.5e-141 & 5.5e-141 & 9.8e-06 & 1.7e-40 & 1.176 & 7.7e-798 & 5.5e-04 & 8.5e-03 \\
28 & 3.1e-1142 & 8.5e-1143 & 4.2e-40 & 1.4e-206 & 5.1e-206 & 8.3e-08 & 9.3e-66 & 1.219 & 1.2e-938 & 5.5e-04 & 8.5e-03 \\
\end{tabular}

%% file: tables/ex4c.tex
\begin{tabular}{llllllllllllllllll}
$j$ %& $\lvert|F\rvert|^-$ 
& $\lvert|F\rvert|^+$ %& $\lvert|u|\rvert^-$ & $\lvert|u|\rvert^+$ 
%& $r^-$ & $r^+$ 
& $q^-$ & $q^+$ & $R^-$ & $R^+$ & $Q^-$ & $Q^+$ & $\delta^-$ & $\delta^+$ 
%& $\zeta^-$ & $\zeta^+$ 
& $\Lambda_1$ & $\Lambda_2^-$ & $\Lambda_2^+$ & $\lvert|E|\rvert$ %& $\varnothing it$ 
& $\text{it}^-$ & $\text{it}^+$ 
%& $\text{conv. to } 0$ & conv. & rem. s. & rem.tot. 
\\\toprule
1 %& 1.6e-1240 
& 1e-1000 %& 9.1e-1241 & 1.7e-1000 & 4.1e-49 & 1.3e-10 
& 1e-240 & 2e-30 & 6e-10 & 0.05 & 4e-88 & 8e+14 & 1.14 & 1.28 %& 6.0e-03 & 1.4e+00 
& 2e-1003 & 2e-11 & 5e-02 & 4e-09 %& 29.4 
& 22 & 32 %& 125000 & 125000 & 0 & 0 
\\
2 %& 1.0e-1230 
& 1e-1000 %& 1.3e-1230 & 1.6e-1000 & 4.5e-50 & 3.0e-15 
& 2e-232 & 8e-35 & 5e-09 & 0.02 & 8e-90 & 8e+08 & 1.13 & 1.27 %& 5.7e-03 & 1.4e+00 
& 1e-1003 & 6e-12 & 5e-04 & 1e-07 %& 25.9 
& 22 & 27 %& 125000 & 125000 & 0 & 0 
\\
3 %& 7.0e-1228 
& 1e-1000 %& 2.0e-1228 & 1.4e-1000 & 4.4e-54 & 1.4e-19 
& 6e-238 & 7e-50 & 5e-11 & 0.003 & 2e-75 & 4e+07 & 1.13 & 1.27 %& 2.1e-03 & 1.4e+00 
& 2e-1004 & 1e-14 & 5e-06 & 5e-10 %& 23.5 
& 21 & 24 %& 125000 & 125000 & 0 & 0 
\\
4 %& 1.0e-1244 
& 1e-1000 %& 7.2e-1245 & 1.6e-1000 & 6e-57 & 5e-21 
& 5e-248 & 4e-59 & 1e-11 & 4e-4 & 1e-92 & 6e+07 & 1.13 & 1.29 %& 3.6e-03 & 1.4e+00 
& 5e-1005 & 5e-16 & 5e-08 & 2e-11 %& 22.0 
& 21 & 23 %& 125000 & 125000 & 0 & 0 
\\
5 %& 3.4e-1251 
& 1e-1000 %& 2.8e-1251 & 1.3e-1000 & 1.4e-57 & 1.4e-22 
& 6e-253 & 1e-63 & 4e-12 & 6e-5 & 6e-104 & 1e+09 & 1.14 & 1.30 %& 2.9e-03 & 1.4e+00 
& 4e-1005 & 3e-18 & 5e-10 & 3e-14 %& 21.0 
& 19 & 21 %& 125000 & 125000 & 0 & 0 
\\\midrule
1 %& 3.5e-1238 
& 1e-1000 %& 4.5e-1238 & 2.7e-1000 & 3.5e-38 & 6.6e-01 
& 1e-241 & 9e-02 & 9e-08 & 0.97 & 9e-90 & 1e+11 & 1.08 & 1.42 %& 6.9e-03 & 1.4e+00 
& 4e-1004 & 4e-10 & 2.2 & 3e-03 %& 34.4 
& 30 & 97 %& 123117 & 123152 & 0 & 1883 
\\
2 %& 2.9e-1224 
& 1e-1000 %& 7.9e-1224 & 3.1e-1000 & 1.2e-41 & 6.9e-11 
& 3e-231 & 1e-29 & 2e-08 & 0.06 & 2e-93 & 6e+12 & 1.13 & 1.27 %& 1.1e-02 & 1.4e+00 
& 1e-1003 & 9e-10 & 5e-02 & 5e-04 %& 29.7 
& 27 & 31 %& 125000 & 125000 & 0 & 0 
\\
3 %& 9.2e-1235 
& 1e-1000 %& 3.2e-1235 & 2.9e-1000 & 1.0e-44 & 4.1e-14 
& 7e-239 & 1e-38 & 4e-09 & 0.016 & 3e-88 & 8e+10 & 1.14 & 1.28 %& 2.1e-02 & 1.4e+00 
& 3e-1003 & 1e-10 & 5e-03 & 8e-05 %& 27.1 
& 25 & 28 %& 125000 & 125000 & 0 & 0 
\\
4 %& 1.1e-1246 
& 1e-1000 %& 3.9e-1247 & 2.8e-1000 & 6.9e-47 & 1.5e-15 
& 4e-250 & 8e-35 & 6e-10 & 0.016 & 4e-99 & 4e+08 & 1.13 & 1.28 %& 3.7e-03 & 1.4e+00 
& 9e-1004 & 6e-12 & 5e-04 & 7e-06 %& 25.8 
& 24 & 26 %& 125000 & 125000 & 0 & 0 
\\
5 %& 7.6e-1203 
& 1e-1000 %& 4.5e-1203 & 3.0e-1000 & 8.1e-49 & 6.9e-18 
& 2e-215 & 5e-40 & 5e-09 & 0.009 & 2e-85 & 4e+08 & 1.13 & 1.27 %& 3e-02 & 1.4e+00 
& 2e-1003 & 2e-12 & 5e-05 & 6e-07 %& 24.1 
& 23 & 25 %& 125000 & 125000 & 0 & 0 
\\\midrule
1 %& 6.4e-1238 
& 1e-1000 %& 6.8e-1238 & 2.0e-1000 & 2.1e-42 & 2.7e-08 
& 2e-240 & 1e-23 & 3e-08 & 0.12 & 1e-94 & 4e+10 & 1.13 & 1.28 %& 4.7e-03 & 1.4e+00 
& 5e-1004 & 3e-10 & 3e-01 & 4e-04 %& 31.9 
& 26 & 35 %& 125000 & 125000 & 0 & 0 
\\
2 %& 2.8e-1237 
& 1e-1000 %& 8.2e-1238 & 2.1e-1000 & 2.1e-45 & 1.1e-11 
& 4e-241 & 1e-34 & 4e-09 & 0.03 & 2e-91 & 9e+09 & 1.13 & 1.27 %& 2.4e-03 & 1.4e+00 
& 4e-1005 & 3e-11 & 2e-02 & 1e-04 %& 28.4 
& 24 & 30 %& 125000 & 125000 & 0 & 0 
\\
3 %& 1.5e-1241 
& 1e-1000 %& 1.3e-1241 & 2.4e-1000 & 1.3e-46 & 1.0e-14 
& 5e-243 & 9e-36 & 1e-09 & 0.02 
& 3e-89 & 1e+09 & 1.13 & 1.28 %& 1.1e-03 & 1.4e+00 
& 4e-1004 & 3e-11 & 2e-03 & 1e-05 %& 26.5 
& 24 & 27 %& 125000 & 125000 & 0 & 0 
\\
4 %& 1.2e-1232 
& 1e-1000 %& 1.1e-1232 & 2.4e-1000 & 2.2e-48 & 7.2e-16 
& 1e-235 & 3e-34 & 2e-09 & 0.02 & 4e-91 & 1e+09 & 1.13 & 1.27 %& 4.3e-03 & 1.4e+00 
& 9e-1004 & 2e-12 & 2e-04 & 2e-06 %& 25.0 
& 23 & 26 %& 125000 & 125000 & 0 & 0 
\\
5 %& 4.3e-1203 
& 1e-1000 %& 9.3e-1203 & 3.0e-1000 & 1.5e-50 & 3.0e-19 
& 5e-212 & 4e-43 & 2e-09 & 0.006 & 2e-74 & 2e+08 & 1.13 & 1.27 %& 4e-03 & 1.4e+00 
& 1e-1003 & 5e-12 & 2e-05 & 2e-07 %& 23.7 
& 22 & 24 %& 125000 & 125000 & 0 & 0 
\\
\end{tabular}